\documentclass[12pt,leqno,a4paper] {amsart}
\usepackage{amssymb,enumerate}
\newcommand{\CC}{{\mathbb{C}}}
\newcommand{\FF}{{\mathbb{F}}}
\newcommand{\QQ}{{\mathbb{Q}}}
\newcommand{\ZZ}{{\mathbb{Z}}}

\newcommand{\bG} {\mathbf G}
\newcommand{\bH} {\mathbf H}
\newcommand{\bL} {\mathbf L}
\newcommand{\bM} {\mathbf M}
\newcommand{\bP} {\mathbf P}
\newcommand{\bS} {\mathbf S}
\newcommand{\bT} {\mathbf T}

\newcommand{\cE} {\mathcal E}
\newcommand{\cH} {\mathcal H}
\newcommand{\cO} {\mathcal O}
\newcommand{\cP} {\mathcal P}
\newcommand{\cS} {\mathcal S}
\newcommand{\cU} {\mathcal U}

\newcommand{\fA} {\mathfrak A}
\newcommand{\fS} {\mathfrak S}
\newcommand{\fP} {\mathfrak P}

\newcommand{\Aut}{{{\operatorname{Aut}}}}
\newcommand{\End}{{{\operatorname{End}}}}
\newcommand{\Out}{{{\operatorname{Out}}}}
\newcommand{\Br}{{{\operatorname{Br}}}}
\newcommand{\Ind}{\operatorname{Ind}}
\newcommand{\Infl}{\operatorname{Infl}}
\newcommand{\IBr}{\operatorname{IBr}}
\newcommand{\Irr}{\operatorname{Irr}}
\newcommand{\Gal}{\operatorname{Gal}}
\newcommand{\GL}{\operatorname{GL}}
\newcommand{\GU}{\operatorname{GU}}
\newcommand{\PGL}{\operatorname{PGL}}
\newcommand{\PSL}{\operatorname{PSL}}
\newcommand{\SL}{\operatorname{SL}}
\newcommand{\SU}{\operatorname{SU}}
\newcommand{\CSp}{\operatorname{CSp}}
\newcommand{\Sp}{\operatorname{Sp}}
\newcommand{\SO}{\operatorname{SO}}
\newcommand{\Spin}{\operatorname{Spin}}
\newcommand{\Syl}{\operatorname{Syl}}
\newcommand{\RLG}{R_{\bL}^\bG}
\newcommand{\RTG}{R_{\bT}^\bG}
\newcommand{\RTH}{R_{\bT}^{\bH}}
\newcommand{\RMG}{{R_\bM^\bG}}
\newcommand{\RML}{{R_\bM^\bL}}
\newcommand{\RLM}{{R_\bL^\bM}}
\newcommand\sRLG{{{}^*\!R_\bL^\bG}}
\newcommand\sRTH{{{}^*\!R_\bT^\bH}}
\newcommand\sRTG{{{}^*\!R_\bT^\bG}}
\newcommand{\tw}[1]{{}^{#1}\!}
\newcommand{\ula}{{\underline{\lambda}}}
\newcommand{\tG}{{\tilde G}}
\newcommand{\tM}{{\tilde M}}
\newcommand{\tchi}{{\tilde\chi}}

\let\vhi=\varphi
\let\la=\lambda
\let\semidi\rtimes

\newtheorem{thm}{Theorem}[section]
\newtheorem{lem}[thm]{Lemma}
\newtheorem{cor}[thm]{Corollary}
\newtheorem{prop}[thm]{Proposition}
\newtheorem{conj}[thm]{Conjecture}
\newtheorem{prob}[thm]{Problem}

\theoremstyle{definition}

\newtheorem{exa}[thm]{Example}

\begin{document}

\title[Local-global conjectures and blocks]{Local-global conjectures\\ and blocks of simple groups}
\date{\today}
\author{Radha Kessar}
\address{Department of Mathematics, City, University of London,
  Northampton Square, EC1V 0HB London UK}
\email{radha.kessar.1@city.ac.uk}
\author{Gunter Malle}
\address{FB Mathematik, TU Kaiserslautern, Postfach 3049,
  67653 Kaisers\-lautern, Germany.}
\email{malle@mathematik.uni-kl.de}

\begin{abstract}
We give an expanded treatment of our lecture series at the 2017 Groups
St Andrews conference in Birmingham on local-global conjectures and the
block theory of finite reductive groups.
\end{abstract}

\maketitle

\section{Introduction}
The aim of these notes is to describe recent progress in the ordinary and
modular
representation theory of finite groups, in particular pertaining to the
local-global counting conjectures. As this relies heavily on having sufficient
knowledge about the representation theory of finite simple groups, we will also
highlight the major advances and results obtained in the block theory of
finite groups of Lie type.

The general setup will be as follows: $G$ will be a finite group,
$$\Irr(G)=\{\text{trace functions of irreducible representations }
G\rightarrow\GL_n(\CC)\}$$
its set of irreducible complex characters. For $\chi\in\Irr(G)$, the value
$\chi(1)$ at the identity element of $G$ is called its \emph{degree}; it is the
degree of any representation affording this character.

We also choose a prime $p$ (with
the interesting case being the one when $p$ divides the group order $|G|$).

The aim is now to link, as much as possible, aspects of the representation
theory of $G$, like its set of irreducible characters $\Irr(G)$, their degrees,
and so on, to those of local subgroups
of $G$. Here a subgroup $N$ of $G$ is called \emph{$p$-local} if $N=N_G(Q)$ for
some $p$-subgroup $1\ne Q\le G$. An important example of local subgroups is
given by the normalisers $N_G(P)$ of Sylow $p$-subgroups $P\in\Syl_p(G)$.

\section{The fundamental conjectures}
The character theory of finite groups was invented by G.~Frobenius more than a
hundred years ago. But still there are many basic open questions that remain
unsolved to the present day. We present some of these in this section.

\subsection{The McKay conjecture}
John McKay in the beginning of the 1970s counted irreducible characters of
odd degree of the newly discovered sporadic simple groups; here are some
such numbers:
$$M_{11}: 4,\ \ M_{12}: 8,\ \ Co_1,Fi_{2}: 32,\ \ B, M: 64.$$
Strikingly, all of these are 2-powers. In 1971 Ian Macdonald \cite{McD} showed
that the number of odd degree irreducible characters is a power of~2 for all
symmetric groups $\fS_n$. (But obviously this statement cannot be
true for all groups, think of the cyclic group of order~3.) The right
generalisation seems to be as follows: let
$$\Irr_{p'}(G):=\{\chi\in\Irr(G)\mid \chi(1)\not\equiv0\pmod p\}$$
be the subset of irreducible characters of $G$ of degree prime to $p$, then the
following should be true \cite{Mc72}:

\begin{conj}[McKay  (1972)]   \label{conj:McK}
 Let $G$ be a finite group, $p$ a prime and $P\in\Syl_p(G)$. Then
 $$|\Irr_{p'}(G)|=|\Irr_{p'}(N_G(P))|.$$
\end{conj}

This conjecture does indeed predict global data in terms of local information.
Since it proposes an explicit formula for $|\Irr_{p'}(G)|$ it is sometimes
also called a counting conjecture.
Note that for $G$ a sporadic group as above, or a symmetric group, Sylow
2-subgroups are self-normalising and then
$\Irr_{2'}(N_G(P))=\Irr_{2'}(P)=\Irr(P/P')$ is an abelian 2-group and hence
of order a power of~2.

Marty Isaacs \cite{Is73} showed in 1973 that Conjecture~\ref{conj:McK} is true
for all groups of odd order, using the Glauberman correspondence, before even
having been aware of McKay's paper.

\begin{exa}   \label{exa:Sn}
Let $G=\fS_n$ the symmetric group of degree~$n$. Frobenius showed how the
irreducible characters of $G$ can be naturally labelled by partitions
$\la\vdash n$ of $n$. Let us write $\chi^\lambda$ for the irreducible
character labelled by $\lambda$. The degree of $\chi^\lambda$ is given by the
well-known hook formula
$$\chi^\lambda(1)=\frac{n!}{\prod_{h}\ell(h)}\,,$$
where the product runs over all \emph{hooks of $\lambda$}, that is, all
boxes $(i,j)$ of the Young diagram of $\lambda$, and $\ell(h)$ denotes the
\emph{length} of the hook in that diagram starting at box $(i,j)$. Macdonald
\cite{McD} determined when this expression is an odd number, and obtained the
following result: write $n=2^{k_1}+2^{k_2}+\ldots$ with $k_1<k_2<\ldots$. Then
$|\Irr_{2'}(\fS_n)|=2^{k_1+k_2+\ldots}$, which is indeed a power of~2. Formulas
for general $p\ge2$ can be given in terms of the $p$-adic expansion of $n$
using generating functions (see \cite{McD}). \par
On the other hand a Sylow $2$-subgroup $P$ of $\fS_n$ is a direct product
$P=P_1\times P_2\times\cdots$ with $P_i=C_2\wr C_2\wr\cdots$ ($k_i$ factors),
an iterated wreath product, and it is self-normalising. Now as already pointed
out above, the only $p'$-characters of $p$-groups are the linear characters,
so that $\Irr_{p'}(P)=\Irr(P/P')$. But
$$|\Irr(P/P')|=|\Irr(P_1/P_1')|\times|\Irr(P_2/P_2')|\times\cdots,$$
and $|\Irr(P_i/P_i')|=|P_i/P_i'|=2^{k_i}$, so indeed we find
$|\Irr_{2'}(P)|=2^{k_1+k_2+\ldots}=|\Irr_{2'}(G)|$ as predicted by McKay's
Conjecture~\ref{conj:McK}.
\end{exa}

While McKay's conjecture is still open, various refinements and extensions
have been proposed; as one example let us mention \cite{IN02}:

\begin{conj}[Isaacs--Navarro (2002)]   \label{conj:IN}
 In the situation of Conjecture~\ref{conj:McK} there exists a bijection
 $\Omega:\Irr_{p'}(G)\rightarrow\Irr_{p'}(N_G(P))$ such that
 $\Omega(\chi)(1)\equiv\pm\chi(1)\pmod p$.
\end{conj}

Paul Fong \cite{F03} showed that this refinement holds for $G=\fS_n$ and all
primes (note that it is stronger than the original McKay conjecture only when
$p\ge5$). Alexandre Turull \cite{Tu06} showed that it holds for all
solvable groups. This was shown to hold for alternating groups by Nath
\cite{N09}

Another such refinement, which is currently being  studied quite intensely, was
proposed by Navarro \cite{Na04}; it proposes that the bijection $\Omega$ should
also be equivariant with respect to $\Gal(\overline{\QQ}_p/\QQ_p)$; see
Brunat and Nath \cite{BN18} for the case of alternating groups, and
Ruhstorfer \cite{Ru18} for groups of Lie type in their defining characteristic.

\subsection{The local-global conjectures}   \label{sec:lg}
The McKay conjecture is concerned with the characters of $p'$-degree.
Now what about characters of degree divisible by $p$? How to relate these to
local data? There is a natural extension of McKay's conjecture, but in order to
formulate this, we need to introduce $p$-blocks. Let $\cO\ge\ZZ_p$ be a big
enough extension, for example containing all $|G|$th roots of unity, and
decompose the group ring of $G$ over $\cO$ into a direct sum of minimal 2-sided
ideals
$$\cO G= B_1\oplus\ldots\oplus B_r,$$
called the \emph{$p$-blocks of $G$}. It is easily seen that this induces a
partition
$$\Irr(G)=\Irr(B_1)\sqcup\ldots\sqcup \Irr(B_r),$$
by decreeing that $\chi\in\Irr(G)$ lies in $\Irr(B_i)$ if and only if
$\chi|_{B_i}\ne0$.
This block subdivision can in fact be read off from the character table of $G$:
$\chi,\psi\in\Irr(G)$ lie in the same $p$-block if and only if
$$\frac{|x^G|\chi(x)}{\chi(1)}\equiv \frac{|x^G|\psi(x)}{\psi(1)}\pmod\fP
  \qquad\text{for all $x\in G$},$$
where $\fP\unlhd\cO$ is the maximal ideal containing $p$.

Richard Brauer showed how to associate to any $p$-block $B$ of $G$ a
$p$-subgroup $D\le G$ of $G$, unique up to conjugation, called \emph{defect
group} of $B$. This can be defined as follows: $D$ is minimal amongst
$p$-subgroups $P$ of $G$ for which there exists a $p'$-element $x$ of $G$ such
that $P$ is a Sylow $p$-subgroup of $C_G(x)$ and 
$$\frac{|x^G|\chi(x)}{\chi(1)}\not \equiv 0 \pmod\fP \qquad
  \text{for all  $\chi \in \Irr(B)$}. $$
Brauer also  constructed a block $b$ of $N_G(D)$ called \emph{Brauer
correspondent} of $B$. The Brauer correspondent of $B$ in $N_G(D)$ is the
unique $p$-block $b$ of $N_G(D)$ with defect group $D$ such that 
$$ \frac{|x^G|\chi(x)}{\chi(1)}\equiv\frac{|x^G|\theta^G(x)}{\theta(1)}\pmod\fP
  \qquad\text{for all $x\in G$, $\chi\in\Irr(B)$ and $\theta\in\Irr(b)$}.$$  

\begin{exa}\label{exa:first}
(a) Let $G$ be a $p$-group. Then $\cO G$ is a single block, with defect group
$D=G$ maximal possible. \par
(b) Let $\chi\in\Irr(G)$ with $\chi(1)_p=|G|_p$, then the corresponding block
$B$ of $G$ has $\Irr(B)=\{\chi\}$ and is called \emph{of defect zero}. Here
$D=1$. For example, if $p$ does not divide $|G|$, then every $p$-block of $G$
is of defect zero. All other blocks contain at least two characters.
For $G=\fS_n$, it is clear from the hook formula in Example~\ref{exa:Sn} that
$\chi^\la$ is of defect zero if and only if $\lambda$ has no
$p$-hook, that is, if and only if $\lambda$ is a \emph{$p$-core}.  \par
(c) The block $B_0$ of $G$ containing the trivial character $1_G$ is called the
principal block of $G$. It always has defect group $D\in\Syl_p(G)$.
\end{exa}

With blocks now at our disposal, McKay's Conjecture~\ref{conj:McK} can be
naturally refined and generalised as follows (see \cite{Al76}):

\begin{conj}[Alperin--McKay (1976)]   \label{conj:AM}
 Let $B$ be a $p$-block of $G$ with defect group $D$ and Brauer correspondent
 $b$ in $N_G(D)$. Then
 $$|\Irr_0(B)|=|\Irr_0(b)|,$$
 where $\Irr_0(B):=\{\chi\in\Irr(B)\mid \chi(1)_p=|G:D|_p\}$.
\end{conj}

The characters in $\Irr_0(B)$ are called \emph{characters of height zero}.
Let us point out that for blocks with defect group $D\in\Syl_p(G)$ we have
$$\Irr_0(B)=\Irr(B)\cap\Irr_{p'}(G),$$
i.e., $\chi\in\Irr(B)$ lies in $\Irr_0(B)$ if and only if $\chi\in\Irr_{p'}(G)$.
It follows that the Alperin--McKay conjecture implies the McKay conjecture,
by just summing over all blocks of full defect.

Again, the Alperin--McKay conjecture gives a local answer to a global question.
It has been proved for all $p$-solvable groups by Okuyama--Wajima and Dade in
1980 \cite{D80,OW80}, for the symmetric groups, the alternating groups and
their covering groups by Olsson \cite{Ol76} and Michler--Olsson \cite{MO90}. 

In view of this conjecture it is of interest to know when it will provide
information on all of $\Irr(B)$, that is, when all characters in $\Irr(B)$ are
of height~0. This is the subject of another even older conjecture by
Brauer \cite{Br55}:

\begin{conj}[Brauer (1955)]   \label{conj:BHZ}
 Let $B$ be a block with defect group $D$. Then
 $$\Irr(B)=\Irr_0(B)\quad\Longleftrightarrow\quad\text{ $D$ is abelian}.$$
\end{conj}

A consequence of this so-called Brauer's Height Zero Conjecture would be an
easy criterion to decide from the character table of a finite group whether
its Sylow
$p$-subgroups are abelian: indeed, any Sylow $p$-subgroup $D\in\Syl_p(G)$ is a
defect group of the principal block $B_0$, and both $\Irr(B_0)$ and
$\Irr_0(B_0)$ are encoded in the character table.

Conjecture~\ref{conj:BHZ} has been proved for $p$-solvable groups by Gluck and
Wolf \cite{GW84}, and for 2-blocks with defect group $D\in\Syl_2(G)$ much more
recently by Navarro and Tiep \cite{NT12} using, among other ingredients,
Walter's classification of groups with abelian Sylow 2-subgroup.

Let us introduce a further fundamental conjecture in this subject. This
purports to count irreducible characters in positive characteristic. For this
let $\IBr(G)$ denote the set of irreducible \emph{$p$-Brauer characters of $G$};
these are lifts to characteristic zero, constructed by Brauer, of the trace
functions of irreducible representations $G\rightarrow\GL_n(\overline{\FF}_p)$.
Again these are partitioned according to the $p$-blocks of $G$, so that
$\IBr(G)=\IBr(B_1)\sqcup\ldots\sqcup\IBr(B_r)$.
A \emph{weight} of $G$ is a pair $(Q,\psi)$ consisting of a $p$-subgroup
$Q\le G$ and an irreducible character $\psi\in\Irr(N_G(Q)/Q)$ of defect zero.
Clearly, $G$ acts on its set of weights by conjugation.
Any weight is naturally attached to a well-defined $p$-block of $G$. Then the
Alperin Weight Conjecture \cite{Al87} proposes:

\begin{conj}[Alperin (1986)]   \label{conj:AWC}
 Let $B$ be a block of $G$. Then
 $$|\IBr(B)|=|\{\text{weights of $G$ attached to $B$}\} \sim_G|.$$
\end{conj}

So, $|\IBr(B)|$ should be determined locally, or more precisely this is the
case whenever $B$ is not of defect zero, since for blocks $B$ of defect zero,
with $\Irr(B)=\{\chi\}$ (see Example~\ref{exa:first}(b)), the corresponding
weight is just $(1,\chi)$. A proof of the Alperin Weight
Conjecture~\ref{conj:AWC} for solvable groups was given by Okuyama \cite{Ok81},
for $p$-solvable groups by Isaacs and Navarro \cite{IN95},
for $\GL_n(q)$ and $\fS_n$ by Alperin and Fong \cite{AF90}, and for groups of
Lie type when $p$ is their defining prime by Cabanes \cite{Ca88}. For blocks
with abelian defect groups (and hence in particular for groups with abelian
Sylow $p$-subgroups), the weight conjecture has the following nice
consequence:

\begin{thm}[Alperin (1986)]   \label{thm:Alperin}
 Let $B$ be a block with abelian defect groups satisfying the Alperin weight
 conjecture, and $b$ its Brauer correspondent. Then:
 $$|\Irr(B)|=|\Irr(b)|\quad\text{and}\quad |\IBr(B)|=|\IBr(b)|.$$
\end{thm}

Kn\"orr and Robinson \cite{KR89} have given reformulations of
Conjecture~\ref{conj:AWC} in terms of chains of $p$-subgroups of $G$. They also
showed the following connection between the conjectures introduced above:

\begin{thm}[Kn\"orr--Robinson (1989)]   \label{thm:KR}
 The following are equivalent for a prime $p$:
 \begin{enumerate}
  \item[\rm(i)] The Alperin--McKay Conjecture~\ref{conj:AM} holds for all
   $p$-blocks with abelian defect;
  \item[\rm(ii)] the Alperin Weight Conjecture~\ref{conj:AWC} holds for all
     $p$-blocks with abelian defect.
 \end{enumerate}
\end{thm}

While all of the above conjectures are open in general, they have been shown to
hold for special types of defect groups. By results of Dade they hold whenever
the
defect group $D$ is cyclic, and by Sambale \cite{Sam14} when $D$ is metacyclic.

Let us mention some further directions which we shall not go into here:
several refinements of the above conjectures have been put forward, like the
Isaacs--Navarro Conjecture~\ref{conj:IN} introduced above. Further, Dade's
conjecture \cite{D92} from 1992 simultaneously generalises the Alperin--McKay
conjecture and the Alperin weight conjecture by making predictions on
characters of arbitrary height. A recent conjecture of Eaton and Moreto
\cite{EM14} extends Brauer's Height Zero Conjecture~\ref{conj:BHZ} to
characters of the first positive height.

\subsection{The reduction approach}
In recent years a new approach for tackling the local-global counting
conjectures has emerged: one tries to study a minimal counterexample by making
use of the classification of the finite simple groups. A first such reduction
in fact dates back quite a while \cite{BK}:

\begin{thm}[Berger--Kn\"orr (1988)]   \label{thm:BK}
 The ``if'' direction of Brauer's Height Zero Conjecture~\ref{conj:BHZ} holds
 if it holds for all blocks of all quasi-simple groups.
\end{thm}

Recall here that a finite group $G$ is \emph{quasi-simple} if $G$ is perfect
and moreover $G/Z(G)$ is simple.
It took 25 years until the necessary statement for quasi-simple groups could
finally be verified, thus giving:

\begin{thm}[Kessar--Malle (2013)]   \label{thm:KM13}
 The ``if'' direction of Brauer's Height Zero Conjecture~\ref{conj:BHZ} holds.
\end{thm}

This result is the outcome of work of many mathematicians on determining all
$p$-blocks of all quasi-simple groups, the case of groups of Lie type
being by far the most challenging. Major contributions are due to
Fong--Srinivasan \cite{FS89}, Brou\'e--Malle--Michel \cite{BMM},
Cabanes--Enguehard \cite{CE99}, Blau--Ellers \cite{BE},
Bonnaf\'e--Rouquier \cite{BR03} and Enguehard \cite{E00}, before the final
case, the so-called quasi-isolated blocks of exceptional groups of Lie type at
bad primes was settled by Kessar and Malle \cite{KM13}.

About 15 years after Berger--Kn\"orr the issue of reductions of the
long-standing conjectures was again taken up by Gabriel Navarro, which led to
the following \cite{IMN}:

\begin{thm}[Isaacs--Malle--Navarro (2007)]   \label{thm:IMN}
 The McKay Conjecture~\ref{conj:McK} holds for a prime $p$ if all finite simple
 groups are McKay good for $p$.
\end{thm}

Here, the reduction is not as clean as for Brauer's height zero conjecture.
The condition of a simple group being McKay good is stronger and more
complicated than just asking that it satisfies the McKay conjecture. We say
that a simple group $S$ of order divisible by~$p$ is \emph{McKay good at $p$}
if the following conditions
hold, where $G$ denotes a \emph{universal covering group of $S$} (that is, $G$
is maximal with respect to being quasi-simple with simple quotient $S$):
\par
Fix $P\in\Syl_p(G)$. There exists a proper subgroup $M<G$ of $G$ with
$N_G(P)\le M$ such that
\begin{enumerate}
\item[(1)] there is a bijection $\Omega:\Irr_{p'}(G)\rightarrow\Irr_{p'}(M)$,
 such that
\item[(2)] $\Omega$ respects central characters, that is, if
 $\chi\in\Irr_{p'}(G)$ lies above $\nu\in\Irr(Z(G))$, then so does
 $\Omega(\chi)$ (note that $Z(G)\le N_G(P)\le M$);
\item[(3)] $\Omega$ is equivariant with respect to
 $\Aut(G)_P=\{\alpha\in\Aut(G)\mid \alpha(P)=P\}$; and
\item[(4)] the Clifford theories in $\Aut(G)_P$ above $\chi$ and $\Omega(\chi)$
 agree (for example, the corresponding 2-cocycles are the same).
\end{enumerate}

The last condition is often the most difficult to check, but it is satisfied
automatically for example if $\Out(G)$ is cyclic; the latter property
holds, for example, for sporadic simple groups, for alternating groups
$\fA_n$ with $n\ne6$, but also for the exceptional groups of Lie type $E_8(q)$.

Let us give some indications on how such a reduction theorem might be arrived
at. The first and crucial step is to generalise the desired assertion:

\begin{conj}[Relative McKay Conjecture]   \label{conj:relMcK}
 Let $G$ be a finite group, $L\unlhd G$, $P/L\in\Syl_p(G/L)$ and suppose that
 $\nu\in\Irr(L)$ is $P$-invariant. Then
 $$|\Irr_{p'}(G|\nu)|=|\Irr_{p'}(N_G(P)|\nu)|.$$
\end{conj}

The original McKay conjecture is recovered from this as the special case when
$L=1$, $\nu=1$. But, despite of seeming to be stronger, the relative version is
much more accessible to an inductive approach. Indeed, Wolf \cite{Wo90} showed
that Conjecture~\ref{conj:relMcK} holds for $p$-solvable groups.
\par
To prove the relative conjecture, let $(G,L,\nu)$ be a minimal counterexample
with respect to $|G/L|$.
\begin{enumerate}
\item[Step 1:] We may assume that $\nu$ is $G$-invariant:\\
Let $L\le T:=G_\nu$ be the stabiliser of $\nu$ in $G$. By assumption $P\le T$,
so $|G:T|$ and $N:N\cap T|$ are prime to $p$. Clifford theory now yields
bijections $\Irr(T|\nu)\rightarrow\Irr(G|\nu)$ and
$\Irr(T\cap N|\nu)\rightarrow\Irr(N|\nu)$ preserving the sets of $p'$-degrees.
Thus, if $T<G$ then $G$ cannot be a minimal counterexample.
\item[Step 2:] We may assume $L\le Z(G)$ is a cyclic $p'$-group and $\nu$ is
faithful:\\
This uses the well-established theory of \emph{character triples}: there exists
a triple $(G^*,L^*,\nu^*)$ with $L^*\le Z(G^*)$ cyclic, $\nu^*\in\Irr(L^*)$
faithful and $G/L\cong G^*/L^*$ such that the Clifford theories in $G$ above
$\nu$ and in $G^*$ above $\nu^*$ agree (see e.g.~\cite[p.~186]{Is73}).
\item[Step 3:] We may assume that $G/L$ has a unique minimal normal subgroup
$K/L$, of order divisible by~$p$:\\
Let $K/L$ be a minimal normal subgroup of $G/L$. Then $|G/K|<|G/L|$. Set
$M:=N_G(KP)$.
If $M<G$ then we may conclude by induction. So we have $M=G$, and hence
$KP\unlhd G$. But then $G/K$ is $p$-solvable. If $G/L$ is $p$-solvable, then
so is $G$, and we may conclude by the theorem of Wolf. Hence, $G/L$ is not
$p$-solvable but $G/K$ is. From this it easily follows that $K/L$ is the unique
minimal normal subgroup.
\item[Step 4:] We are done if the simple composition factors of $K/L$ are McKay
good for $p$:\\
This is by far the most difficult part of the argument in \cite{IMN}, and we
will not go into it here.
\end{enumerate}
A streamlined version of the arguments for this and more general reductions
has been published by Sp\"ath \cite{Sp17b}.

Since the publication of Theorem~\ref{thm:IMN} all conjectures introduced above
have been shown to reduce to properties of simple groups:
\begin{enumerate}
\item the Alperin--McKay Conjecture~\ref{conj:AM} holds if all simple groups
 are \emph{AMcK good} (Sp\"ath \cite{Sp13a});
\item the Alperin Weight Conjecture~\ref{conj:AWC} holds if all simple groups
 are \emph{AWC good} (Na\-varro--Tiep \cite{NT11}, and Sp\"ath \cite{Sp13b} for
 the blockwise version);
\item the ``only if'' direction of Brauer's Height Zero
 Conjecture~\ref{conj:BHZ} holds if all simple groups
 are AMcK good and moreover it holds for all quasi-simple groups
 (Navarro--Sp\"ath \cite{NS14}).
\end{enumerate}
The assertion on quasi-simple groups necessary for BHZ was subsequently shown
by Kessar--Malle \cite{KM17}.

Moreover, for blocks $B$ with abelian defect groups, Koshitani and Sp\"ath
\cite{KS15a} show that being AWC good is implied by being AMcK good, if
moreover the $p$-modular decomposition matrix of $G$ is lower uni-triangular
with respect to an $\Aut(G)$-stable subset of characters.
Thus the reductions have also uncovered some remarkably strong connections
between the various conjectures.

We will not endeavour to spell out the somewhat technical conditions for being
good in the various cases, let us just say that they are similar to the one
of being McKay good explained above. See \cite{Sp17b}, for example.
Sp\"ath \cite{Sp17a} has also succeeded in reducing Dade's conjecture to a
property of simple groups.

So now all of the conjectures have been reduced to questions on finite simple
groups, can we solve them? Well, it turns out that our knowledge on the
representation theory of quasi-simple groups is not yet well-developed enough
to really answer these questions. Roughly speaking, the alternating groups
can be treated combinatorially (see also Example~\ref{exa:Sn-local} for an
illustration in symmetric groups), extending the aforementioned results of
Olsson and Alperin--Fong to accommodate the stronger inductive conditions 
(see Denoncin \cite{De14}), the sporadic simple groups can be treated by ad hoc
case-by-case methods
(and this has been completed by various authors, except for the Alperin weight
conjecture for the very largest sporadic groups, see e.g. An and Dietrich
\cite{AD12} and Breuer \cite{BrWeb}). Similarly, the case of the finitely
many exceptional covering groups of the simple groups of Lie type have been
settled. Thus we are left with the by far biggest class of examples: the
16 infinite families of simple groups of Lie type. Here, the case when $p$
is the defining prime has been shown to hold by Sp\"ath for all conjectures
\cite{Sp12,Sp13a,Sp13b} building on previous work of Maslowski \cite{Ms10}.

For the rest of these lectures we will concentrate on the McKay conjecture for
groups of Lie type. First we
need to understand the sets $\Irr(G)$ and $\Irr(N_G(P))$, or $\Irr(M)$ for
a suitable proper subgroup $N_G(P)\le M<G$.

\subsection{McKay's conjecture for $\GL_n(q)$}
Let's take a look at the case of $G=\GL_n(q)$, $q=p^f$ a prime power. Here,
the ordinary character table was determined by Green \cite{Gr55} in 1955. We
need two ingredients. First let
$$B=\Big\{\begin{pmatrix} *& \ldots& *\\ & \ddots\\ 0& & *\end{pmatrix}\Big\}
  \le G$$
be the Borel subgroup of $G$ consisting of upper triangular invertible
matrices, and consider the induced character $1_B^G$ (the permutation character
of $G$ on the cosets of $B$). 

\begin{thm}[Green (1955)]   \label{thm:Borel}
 The constituents of $1_B^G$ are in bijection with partitions $\la\vdash n$
 such that
 $$1_B^G=\sum_{\lambda\vdash n}\chi^\la(1)\,\rho_q^\la,$$
 where $\rho_q^\la\in\Irr(\GL_n(q))$ is the character labelled by $\la$,
 and $\chi^\la\in\Irr(\fS_n)$ is as in Example~\ref{exa:Sn}.
\end{thm}

Thus, the permutation character of $G$ on $B$ decomposes similarly to the
regular character of $\fS_n$. This result is one reason why $\GL_n(q)$ is
sometimes called a ``quantisation of $\fS_n$'', or ``$\fS_n=\GL_n(1)$''. The
proof of Theorem~\ref{thm:Borel} rests on the fact that the endomorphism
algebra $\End_{\CC G}(1_B^G)$ is an Iwahori--Hecke algebra $\cH(\fS_n,q)$ at
the parameter $q$, which by Tits' deformation theorem is isomorphic to the
complex group algebra $\CC\fS_n$ of $\fS_n$. The constituents $\rho_q^\la$,
$\la\vdash n$, of $1_B^G$ occurring in Theorem~\ref{thm:Borel} were later
called the \emph{unipotent characters} of $G$.

Now let $s\in\GL_n(q)$ be a $p'$-element; then $s$ is diagonalisable over
a finite extension of $\FF_q$ (it is a \emph{semisimple element} of $G$). Its
characteristic polynomial has the form $\prod_{i=1}^r f_i^{n_i}$ with
suitable irreducible polynomials $f_i\in\FF_q[X]$ of degrees $d_i=\deg(f_i)$
such that $\sum_i n_id_i=n$. Then
$$C_{\GL_n(q)}(s)\cong\GL_{n_1}(q^{d_1})\times\cdots\times\GL_{n_r}(q^{d_r}).$$
Now for partitions $\la_i\vdash n_i$, $1\le i\le r$, and
$\rho_{q^{d_i}}^{\la^i}$ the corresponding unipotent characters of the factors
$\GL_{n_i}(q^{d_i})$, we have an irreducible character
$\rho_{q^{d_1}}^{\la^1}\otimes\cdots\otimes\rho_{q^{d_r}}^{\la^r}$ of $C_G(s)$.
Let us write $\cS$ for the set of all pairs $(s,\ula)$, where $s\in \GL_n(q)$
is a semisimple element up to conjugation with characteristic polynomial
$\prod f_i^{n_i}$, and $\ula=(\la^1,\ldots,\la^r)\vdash(n_1,\ldots,n_r)$ is
an $r$-tuple of partitions.

\begin{thm}[Green (1955)]   \label{thm:Green}
 There is a natural bijection
 $$\cS\longrightarrow\Irr(\GL_n(q)),\qquad (s,\ula)\mapsto\rho^{s,\ula},$$
 such that
 $$\rho^{s,\ula}(1)
   =|\GL_n(q):C_{\GL_n(q)}(s)|_{p'}\cdot\prod_{i=1}^r\rho_{q^{d_i}}^{\la^i}(1).
 $$
\end{thm}

The sets $\cE(G,s):=\{\rho^{s,\ula}\}\subseteq\Irr(G)$ are  called
\emph{Lusztig series}.
Observe that by its definition $\cE(G,s)$ is in bijection with the set of
$r$-tuples $\{(\la^1,\ldots,\la^r)\vdash(n_1,\ldots,n_r)\}$ of
partitions, which in turn parametrise the unipotent characters of
$C_G(s)=\GL_{n_1}(q^{d_1})\times\cdots\times\GL_{n_r}(q^{d_r})$. This is 
called the \emph{Jordan decomposition} of the characters in $\Irr(\GL_n(q))$.

Thus, in order to verify for example the McKay conjecture we need to know the
unipotent character degrees. These turn out to be given by a quantisation
of the hook formula that we already saw for the character degrees of $\fS_n$
in Example~\ref{exa:Sn}
$$\rho_q^\la(1)=q^{a(\la)}\,\frac{[n]_q!}{\prod_h[\ell(h)]_q}\,,$$
where the product runs again over all hooks $h$ of $\la$. Here,
$[m]_q:=(q^m-1)/(q-1)$ for $m\ge1$, $[n]_q!:=[1]_q\cdots[n]_q$, and
$a(\la):=\sum_i(i-1)\la_i$ when $\la=(\la_1\ge\la_2\ge\ldots)$.

There are two cases to discuss for the McKay conjecture: either the relevant
prime equals $p$, or it is different from $p$, in which case we will call
it $\ell$. Let us first prove the following:

\begin{prop}   \label{prop:p'}
 We have
 $$\Irr_{p'}(\GL_n(q))=\{\rho^{s,\ula}\mid \ula=((n_1),\ldots,(n_r)),\
   s\in\GL_n(q)\text{ semisimple}\},$$
 so $\Irr_{p'}(\GL_n(q))$ is in bijection with the semisimple conjugacy classes
 of $\GL_n(q)$.
\end{prop}

\begin{proof}
By the degree formula given above, $p$ does not divide $\rho^{s,\ula}(1)$
if and only if $\sum a(\la^i)=0$, that is, if and only if all partitions
$\la^i$ are of the form $\la^i=(n_i)$.
\end{proof}

Let us now consider the local side. Here
$$P=\Big\{\begin{pmatrix} 1& & *\\ & \ddots\\ 0& & 1\end{pmatrix}\Big\}
  \le G$$
is a Sylow $p$-subgroup of $G$, and $N_G(P)=B=P.T$ with
$$T=\Big\{\begin{pmatrix} *& & 0\\ & \ddots\\ 0& & *\end{pmatrix}\Big\}
  \le G$$
an abelian subgroup (a so-called \emph{maximally split maximal torus of $G$}).
Thus
$$\Irr_{p'}(N_G(P))=\Irr_{p'}(B)=\Irr_{p'}(P.T)=\Irr_{p'}(P/P'.T).$$
Maslowski \cite{Ms10} showed that $\Irr_{p'}(P/P'.T)$ can be parametrised by
$\FF_q^\times\times(\FF_q)^{n-1}$, and thus by the set of monic polynomials of
degree~$n$ over $\FF_q$ with non-vanishing constant coefficient, which are
exactly the possible characteristic polynomials of semisimple elements in
$\GL_n(q)$. By Proposition~\ref{prop:p'} this establishes McKay's conjecture
for $\GL_n(q)$ and the prime $p$. He also showed that the constructed bijection
is equivariant with respect to $\Aut_P(G)$ (the relevant outer automorphisms
are field automorphisms and the transpose-inverse automorphism when $n>2$),
and also compatible with respect to central characters.

Now what about primes $\ell\ne p$?
Here we have the following observation, which is again immediate from the
degree formula:

\begin{prop}   \label{prop:ell'GL}
 Let $\ell\ne p$. Then $\rho^{s,\ula}\in\Irr_{\ell'}(\GL_n(q))$ if and only if
 $s$ centralises a Sylow $\ell$-subgroup of $\GL_n(q)$ and
 $\rho_{q^{d_i}}^{\la^i}\in\Irr_{\ell'}(\GL_{n_i}(q^{d_i}))$ for all
 $i=1,\ldots,r$.
\end{prop}

That is, in order to prove the McKay conjecture in this case we are reduced
to understanding the unipotent characters $\rho^{\la^i}$, and for this, to
determine when $\ell$ does divide a factor $q^m-1$. Let $\Phi_d$ denote the
$d$th cyclotomic polynomial over $\QQ$, so that
$q^m-1=\prod_{d|m}\Phi_d(q)$. We also write $d_\ell(q)$ for the order of $q$
in $\FF_\ell^\times$, that is, its order modulo~$\ell$. We have the following
elementary criterion:

\begin{lem}   \label{lem:div}
 Let $d=d_\ell(q)$. Then $\ell$ divides $\Phi_d(q)$ if and only if
 $m\in\{d,d\ell,d\ell^2,\ldots\}$.
\end{lem}

Then the degree formula implies:

\begin{cor}   \label{cor:hooks}
 Let $\ell>2$. For $\la\vdash n$ we have $\rho^\la\in\Irr_{\ell'}(\GL_n(q))$ if
 and only if $\la$ has exactly $w$ hooks of length~$d$, where $d=d_\ell(q)$ and
 $n=wd+r$ with $0\le r<d$.
\end{cor}

Let's now turn to the local situation.

\begin{prop}   \label{prop:ell-local}
 Let $d=d_\ell(q)$ and write $n=wd+r$ with $0\le r<d$. The normaliser $N_G(P)$
 of a Sylow $\ell$-subgroup $P$ of $G:=\GL_n(q)$ is contained in
 $$N_G(\GL_1(q^d)^w)=\big((\GL_1(q^d).C_d)\wr\fS_w\big)\times\GL_r(q)
   \cong\big(\GL_1(q^d)^w.G(d,1,w)\big)\times\GL_r(q),$$
 where $G(d,1,n)= C_d\wr\fS_w$ is an imprimitive complex reflection group.
\end{prop}

The structure of a Sylow $\ell$-normaliser is quite complicated in general,
but by the Reduction Theorem~\ref{thm:IMN} we can instead consider the
intermediate group $M:=\big(\GL_1(q^d)\wr G(d,1,w)\big)\times\GL_r(q)$ which is
much closer to being a finite reductive group like $\GL_n(q)$ itself. Now
taking together \cite{Ma07} and the result of Sp\"ath \cite{Sp10} show:

\begin{thm}[Malle, Sp\"ath (2010)]   \label{thm:MS10}
 Let $G=\GL_n(q)$ and $M$ as above. There is a bijection
 $$\Irr_{\ell'}(G)\rightarrow\Irr_{\ell'}(M).$$
\end{thm}

The proof relies on combinatorial descriptions of the sets on both sides that
coincide.

This sketch shows how to find a McKay bijection in the case of $\GL_n(q)$. A
very similar statement also holds for the general unitary groups, but the
proofs are different and more complicated.

Now the general linear groups are in general not quasi-simple; the right
groups to consider are the special linear groups $\SL_n(q)$, where
unfortunately the situation is much less transparent. Still, Cabanes and
Sp\"ath \cite{CS15b} showed how to descend the bijection from
Theorem~\ref{thm:MS10} to $\SL_n(q)$ and thus show that this group is McKay
good for all $\ell$.

\subsection{Groups of Lie type}   \label{subsec:lietypegroups}
We now turn to general groups of Lie type. A finite group $G$ is said to be of
Lie type if $ G=\bG^F$, where $\bG$ is a connected reductive  group  over an
algebraic closure of a finite field with a Frobenius map $F:\bG\rightarrow\bG$
(we refer to \cite{MT} for an introduction to the structure theory of these
groups). A subgroup $\bH $ of $\bG$ is said to be  $F$-stable if $F(h)\in\bH$
for all $h\in\bH$. If $\bH\leq\bG$ is $F$-stable, then $\bH^F$ is a finite
group. A \emph{torus of $\bG$} is a closed connected abelian subgroup of $\bG$
consisting of semisimple elements. The group $\bG$ acts on the  set of tori of
$\bG$;  maximal tori form a single $\bG$-orbit. The group $\bG^F$ acts on the
set of $F$-stable tori of $\bG$ but there is in general more than one
$\bG^F$-orbit  
of $F$-stable maximal tori of $\bG$ and $F$-fixed point subgroups of tori in
different $\bG^F$-classes have different orders. The  $\bG^F$ classes of
maximal tori can be described using Weyl groups. Fix an $F$-stable maximal
torus $\bT$ of $\bG$ and set $W(\bT)= N_\bG(\bT)/\bT$, the Weyl group of $\bT$.
The $F$-action on $\bT$ induces an $F$-action on $W(\bT )$. Elements $w$ and
$w'$ of $W(\bT)$ are said to be $F$-conjugate if $w'=x w F(x)^{-1}$ for some
$x\in W(\bT)$. The $\bG^F$-conjugacy classes of $F$-stable maximal tori are
in bijection with the $F$-conjugacy classes of $W$ which in turn can be
described combinatorially.

\begin{exa} \label{exa:generallineartori} 
Let $\bG =\GL_n(\overline{\FF}_q)$ and $F :\bG\to\bG$ be the standard Frobenius
map which sends every entry of a matrix in $\bG$ to its $q$-th power. Then
$G =\bG^F=\GL_n(q)$. Let $\bT$ be the subgroup of $\bG$ consisting of all
diagonal matrices. Then $\bT$ is an $F$-stable maximal torus of $\bG$. It is
easy to see that $N_\bG(\bT)$ consists of all monomial matrices and
$W(\bT)\cong\fS_n$. Further, since $F$ fixes every permutation matrix the
induced action of $F$ on $W(\bT)$ is trivial. So, the  $F$-conjugacy classes of
$W(\bT)$ are simply the conjugacy  classes of $ W(\bT)$. We obtain a bijection 
$$ \{\text{partitions of $n$}\}\stackrel{\rm {1:1}}{\longrightarrow}
  \{ \text{$F$-stable maximal tori of $\bG $}\} /\bG^F. $$
If $\la =(\la_1,\ldots,\la_s)\vdash n$ corresponds to the  $\bG^F$-class of the
maximal torus $\bT_{\la} $ then   
$$ |\bT^F_{\la}|  = ( q^{\la_1} -1)\cdots  (q^{\la_s}-1). $$ 
 \end{exa} 
 
In their seminal 1976 paper \cite{DeLu76}, Deligne and Lusztig showed how to
construct ordinary $\bG^F$-representations from the $\ell$-adic cohomology
spaces of certain algebraic varieties (now called Deligne--Lusztig varieties)
on which $\bG^F$ acts. For each  $F$-stable maximal torus $\bT$ of $\bG$,
they constructed a pair of $\ZZ$-linear maps
$$\RTG:\ZZ\Irr(\bT^F)\to\ZZ\Irr(\bG^F),\qquad
  \sRTG:\ZZ \Irr(\bG^F)\to\ZZ\Irr(\bT^F).$$
The maps $\RTG$ and $\sRTG$ are called the \emph{Deligne--Lusztig induction and
restriction maps}  respectively. These maps are adjoint to each other with
respect to the standard scalar product on the space of class functions of
$\bG^F$, that is for each $\chi\in\Irr(\bG^F)$, $\theta\in\Irr(\bT^F)$,
$$\langle\chi,\RTG(\theta)\rangle=\langle\sRTG(\chi),\theta\rangle.$$  

For an $F$-stable maximal torus $\bT$ of $\bG$ and $\theta $ an irreducible
character of $ \bT^F$, let $\cE(\bG^F|(\bT,\theta))$ be the subset of
$\chi\in\Irr(\bG^F)$ consisting of those $\chi$ such that
$\langle\chi,\RTG(\theta)\rangle\ne0$. The group $\bG^F$ acts by conjugation
on the set of pairs $(\bT,\theta)$ where $\bT$ is an $F$-stable maximal torus
of $\bG$ and $\theta$ is an irreducible character of $\bT^F$ and this action
preserves the sets $\cE(\bG^F|(\bT, \theta))$, that is for all $g\in\bG^F$,
$\bT$ an $F$-stable maximal torus of $\bG$ and $\theta\in\Irr(\bT^F)$, 
$$\cE(\bG^F|\,^g(\bT,\theta)) = \cE(\bG^F|(\bT,\theta)).$$
The virtual characters $\RTG(\theta)$ as $(\bT,\theta)$ runs over all pairs of
$F$-stable maximal tori $\bT$ of $\bG$ and irreducible characters $\theta$ of
$\bT^F$ ``trap'' all irreducible characters of $\bG^F$:

\begin{thm} [Deligne--Lusztig (1976)]  
 $$\Irr(\bG^F) = \bigcup_{(\bT,\theta)} \cE(\bG^F|(\bT, \theta) )$$
 as $(\bT,\theta)$ runs over the $\bG^F$ conjugacy classes of pairs
 $(\bT, \theta)$ where $\bT$ is an $F$-stable maximal torus of $\bG$ and
 $\theta $ is an irreducible character of $\bT^F$.
\end{thm} 

\subsection{Characters of groups of Lie type}
Let $\bG$ be connected reductive with a Frobenius map $F:\bG\rightarrow\bG$.
If $\bG$ is simple of simply connected type (like, for example, $\bG=\SL_n$),
then $\bG^F$ is, apart from a few
exceptions, a finite quasi-simple group of Lie type. Moreover, all such groups,
except for the Ree and Suzuki groups for which a slightly more general setup is
needed, are obtained in this way. This turns out to be the right setting to
study the character theory of the families of groups of Lie type.

Recall that for $\bT\le\bG$ an $F$-stable maximal torus, and
$\theta\in\Irr(\bT^F)$ there is an associated virtual Deligne--Lusztig
character $\RTG(\theta)$. As for $\GL_n(q)$ the set of irreducible
characters of $G=\bG^F$ can be partitioned into Lusztig series, as follows.
Define a graph on $\Irr(G)$ by connecting two characters
$\chi,\chi'\in\Irr(G)$ if there exists a pair $(\bT,\theta)$ such that
$\langle\chi,\RTG(\theta)\rangle\ne0\ne\langle\chi',\RTG(\theta)\rangle$.
The connected components of this graph are the \emph{Lusztig series in
$\Irr(G)$}. This also defines an equivalence relation on the set of pairs
$(\bT,\theta)$ which seems a bit mysterious. Lusztig has shown that the
Lusztig series can instead also be parametrised by semisimple classes of
a group $G^*$ closely related to $G$. It is obtained from the Langlands
dual group $\bG^*$ of $\bG$ as fixed points under a Frobenius map that we will
also denote
by $F$. Here, the Langlands dual $\bG^*$ has root datum obtained from that of
$\bG$ by exchanging character group and cocharacter group. For example,
$$\GL_n^*=\GL_n,\ \SL_n^*=\PGL_n,\ \Sp_{2n}^*=\SO_{2n+1},\ E_8^*=E_8,\ldots$$
One usually writes $\cE(G,s)\subseteq\Irr(G)$ for the Lusztig series indexed by
$s\in G^*$.

\begin{exa}
Let $G=\GL_n(q)$, $s\in G^*=\GL_n(q)$ semisimple. Let $\bT\le C_\bG(s)$ be an
$F$-stable maximal torus. Then $s\in T=\bT^F$ corresponds to
some $\theta\in\Irr(T)$ under the isomorphisms $T\cong\Irr(T)$ induced by the
duality between $\bG$ and $\bG^*$. Then
$$\cE(G,s)=\bigcup_{T,\theta} \cE(\bG^F|(\bT,\theta)),$$
the union running over all such pairs $(T,\theta)$.
\end{exa}

In particular when $s=1$ then all tori $T$ contain~$s$, and $s$ corresponds to
the trivial character $1_T$ of $T$, so
$$\cE(G,1)=\bigcup_T\cE(\bG^F|(\bT,1_T))$$
and these are the \emph{unipotent characters of $G$}. Lusztig has shown that
they are parametrised independently of $q$ by suitable combinatorial data only
depending on the complete root datum (the type) of $(\bG,F)$. For example, we
had already seen that for $G=\GL_n(q)$, the unipotent characters are
parametrised by partitions of $n$, independently from $q$. As for $\GL_n(q)$
there is a Jordan decomposition, which we state here only in a special
situation, see \cite{LuB}:

\begin{thm}[Lusztig (1984)]   \label{thm:Jor}
 Assume that $s\in G^*$ is such that $C_{\bG^*}(s)$ is connected. Then there is
 a bijection
 $$J_s:\cE(G,s)\longrightarrow\cE(C_{G^*}(s),1),$$
 with
 $$\chi(1)=|G^*:C_{G^*}(s)|_{p'}\cdot J_s(\chi)(1).$$
\end{thm}

Lusztig called this bijection the \emph{Jordan decomposition} of irreducible
characters. The assumption on $s$ is satisfied for example for all semisimple
elements in $\GL_n(q)$, and more generally for all semisimple elements in
groups $\bG$ with connected centre, like, for example, $\PGL_n$ or $E_8$.

\begin{exa}
Let $s\in G^*$ be such that $C_{\bG^*}(s)=\bT^*$ is a maximal torus of $\bG^*$.
The element $s$ is then called \emph{regular}. Regular semisimple elements are
dense in $\bG^*$, so this assumption is satisfied for ``most'' elements. In
this case
$|\cE(G,s)|=1$, and $\chi(1)=|G^*:T^*|_{p'}$ for $\{\chi\}=\cE(G,s)$.
\end{exa}

If $C_{\bG^*}(s)$ is disconnected, the situation is considerably more
complicated, but still Lusztig obtained an analogue of Jordan decomposition
\cite{Lu88}.

\begin{exa}
Let $G=\SL_2(q)$ with $q$ odd, so $G^*=\PGL_2(q)$. The semisimple elements in
$G^*$ are: the trivial element, two classes of elements of order~2 with
disconnected centraliser (one lying inside $\PSL_2(q)$, one outside), and all
other semisimple elements are regular with centraliser of order either $q-1$ or
$q+1$. Letting $s_1,s_2$ denote representatives of the two classes of
involutions we thus have
$$\Irr(G)=\cE(G,1)\cup\cE(G,s_1)\cup\cE(G,s_2)\cup
          \bigcup_{s: s^2\ne1}\cE(G,s),$$
where $|\cE(G,1)|=|\cE(G,s_i)|=2$, $|\cE(G,s)|=1$; here $s_1,s_2$ are
representatives of the two classes of involutions.
\end{exa}

\subsection{Towards McKay's conjecture for groups of Lie type}
Again it is straightforward from the Jordan decomposition to classify the
characters in $\Irr_{\ell'}(G)$:

\begin{prop}  \label{prop:ell'}
 Let $\chi\in\cE(G,s)$. Then $\chi\in\Irr_{\ell'}(G)$ if and only if $s$
 centralises a Sylow $\ell$-subgroup of $G^*$ and moreover
 $J_s(\chi)\in\cE(C_{G^*}(s),1)$ is contained in $\Irr_{\ell'}(C_{G^*}(s))$.
\end{prop}

So again our question is reduced to studying unipotent characters. Their
degrees are given by polynomial expressions in the field size~$q$, as we
already saw for $\GL_n(q)$ with the hook formula. It is combinatorially easy
to determine the $\ell'$-degrees from this for classical types; for
exceptional types this is just a finite task.

Now let's turn again to the local picture. We set $d=d_\ell(q)$, where we
recall that $d_\ell(q)$ denotes the order of $\ell$ modulo $q$. Assume  
for simplicity that $\ell\ne2$. We describe the picture for $\bG$ of classical
type, that is $G=G_n(q)=\Sp_{2n}(q)$, $\SO_{2n+1}(q)$ or $\SO_{2n}^\pm(q)$.
First assume that $d$ is odd and write $n=ad+r$ with $0\le r<d$. Then there is
a torus $T_d=\GL_1(q^d)\times\cdots\times\GL_1(q)$ ($w$ factors) of $G$ such
that
$$N_G(P)\le N_G(T_d)=T_d.(C_{2d}\wr\fS_w)\times G_r(q)$$
contains the normaliser of a Sylow $\ell$-subgroup $P$ of $G_n(q)$ (see
\cite[\S3.2]{BMM}).

\begin{exa}
Assume $d=1$ and $G=G_n(q)\ne\SO_{2n}^-(q)$. Then $T_d\cong C_{q-1}^n$ is a
maximally split torus, with $N_G(T_d)=T_d.W$, with $W$ the Weyl group of $G$.
\end{exa}

If instead $d=2e$ is even, then the same type of result holds, we just have to
replace the cyclic group $\GL_1(q^d)=C_{q^d-1}$ by the cyclic group
$\GU_1(q^e)=C_{q^e+1}$.

\begin{thm}[Malle, Sp\"ath (2010)]   \label{thm:MS10general}
 Let $\bG$ be simple of simply connected type, $\ell\ne p$ a prime, and
 $d=d_\ell(q)$. Then there
 exists a bijection $\Omega:\Irr_{\ell'}(G)\rightarrow\Irr_{\ell'}(N_G(T_d))$
 with $\Omega(\chi)(1)\equiv\pm\chi(1)\pmod\ell$ for all $\chi$, unless one
 of
 \begin{enumerate}
  \item[$\bullet$] $\ell=3$, $G=\SL_3(q)$, $\SU_3(q)$, or $G_2(q)$ with
   $q\equiv2,4,5,7\pmod9$, or
  \item[$\bullet$] $\ell=2$, $G=\Sp_{2n}(q)$ with $q\equiv3,5\pmod8$.
 \end{enumerate}
 This bijection can be chosen to preserve central characters.
\end{thm}

The proof is obtained by parametrising both sides by the same combinatorial
data.

For the listed exceptions $N_G(T_d)$ does not even contain a Sylow
$\ell$-subgroup; for example when $G=\Sp_2(q)\cong\SL_2(q)$ with
$q\equiv3,5\pmod8$ a Sylow 2-subgroup is quaternion and thus cannot be
contained in $N_G(T_d)$ which is an extension of a cyclic group of order
$q\pm1$ by a group of order~2.

Still the exceptions were shown to be McKay good \cite{Ma08b}. Note that the
group $N_G(T_d)$ only depends on $d$, but not on $\ell$.
Theorem~\ref{thm:MS10general} also
gives the Isaacs--Navarro refinement from Conjecture~\ref{conj:IN}.

Now, what's missing for proving McKay goodness? Equivariance and Clifford
theory!
Recall: for $G$ quasi-simple of Lie type, $\Out(G)$ is made up of diagonal,
graph and field automorphisms (see e.g.\cite[Thm.~24.24]{MT}):
\begin{enumerate}
\item diagonal automorphisms are induced e.g. by the embedding of $\SL_n(q)$
 into $\GL_n(q)$, or $\Sp_{2n}(q)$ into $\CSp_{2n}(q)$,
\item graph automorphisms come from the Dynkin diagram (e.g., the
 transpose-inverse automorphism for $\SL_n(q)$, $n\ge3$, or triality on
 $D_4(q)$),
\item field automorphisms come from the field $\FF_q$ over which $G$ is
 defined.
\end{enumerate}

\begin{exa}
The worst case, in the sense that the structure of the outer automorphism
group is most complicated, occurs for $G=\Spin_8^+(q)$ with $q\equiv1\mod2$;
here $\Out(G)=2^2.\fS_3.C_f$, where $q=p^f$. \par
Nice cases (with small outer automorphism group) are, by contrast,
$G=E_8(q)$ or $G=\Sp_{2n}(q)$ with $q$ even; here $\Out(G)=C_f$ is cyclic.
\end{exa}

\begin{thm}[Cabanes--Sp\"ath (2013)]   \label{thm:CS13}
 Let $S$ be simple of Lie type such that $\Out(S)$ is cyclic, then the
 bijection in Theorem~\ref{thm:MS10general} can be made equivariant. In
 particular, $S$ is then McKay good.
\end{thm}

In general, we need to solve the following hard problem:

\begin{prob}
 For $G$ quasi-simple of Lie type, determine the action of $\Aut(G)$ on
 $\Irr(G)$.
\end{prob}
 
Partial results are available: Lusztig determined the action of diagonal
automorphisms: they leave $\chi\in\cE(G,s)$ invariant unless possibly when
$C_{\bG^*}(s)$ is disconnected.

Also, the action of all automorphisms is known on Lusztig series where
 $C_{\bG^*}(s)$ is connected.

\begin{exa}[Lusztig]
Consider the case of unipotent characters $\chi\in\cE(G,1)$. If $G$ is
quasi-simple then any automorphism of $G$ fixes all unipotent characters,
unless $G$ is of type $D_{2n}$, or $B_2$, $F_4$ in characteristic~2, or
type $G_2$ in characteristic~3 (see e.g. \cite{Ma08b}).
\end{exa}

Let $\bG\hookrightarrow\tilde\bG$ be a \emph{regular embedding}, that is
$\tilde\bG$ is a connected reductive group with a connected center and the same
derived subgroup as $\bG$. For example, the inclusion of $\SL_n$ in $\GL_n$
is a regular embedding. We assume that the Frobenius endomorphism $F$ extends
to a Frobenius morphism, also denoted $F$, of $\tilde\bG$. The action of
$\tilde G=\tilde\bG^F$ on $G$ is by inner-diagonal automorphisms. Denote by
$D$ the group of graph and field automorphisms of $\tilde G$. Then Sp\"ath
\cite{Sp12} showed the following:

\begin{thm}[Criterion of Sp\"ath]   \label{thm:crit}
 Assume there is an $\Aut(G)_P$-equivariant bijection
 $\tilde\Omega:\Irr_{\ell'}(\tG)\rightarrow\Irr_{\ell'}(\tM)$ compatible with
 multiplication by $\Irr(\tilde G/G)$. If
 \begin{enumerate}
  \item[$\bullet$] for every $\tchi\in\Irr_{\ell'}(\tG)$ there is
   $\chi\in\Irr_{\ell'}(G|\tchi)$ with
   $$(\tG\semidi D)_\chi=\tG_\chi\semidi D_\chi$$
   and $\chi$ extends to $(G\semidi D)_\chi$, and
  \item[$\bullet$] the analogous condition holds on the local side,
 \end{enumerate}
 then $G/Z(G)$ is McKay good for $\ell$.
\end{thm}
 
So, for $G=\SL_n(q)$, for example, one uses the bijection for $\tG=\GL_n(q)$,
then has to check the stabiliser condition and finally prove extendibility.
This leads to the following situation at the time of writing:
$S$ simple group is McKay good for all primes, unless possibly when $S$ is of
type $B_n(q)$, $\tw{(2)}D_n(q)$, $\tw{(2)}E_6(q)$ or $E_7(q)$.

There is one prime for which more can be said: $\ell=2$.

\begin{thm}[Malle--Sp\"ath (2016)]   \label{thm:MS17}
 Let $G$ be quasi-simple of Lie type, not of type $A$, and
 $\chi\in\Irr_{2'}(G)$. Then there exists a linear character $\theta\in\Irr(B)$
 where $B\le G$ is a Borel subgroup, such that $\chi$ is a constituent of
 $\Ind_B^G(\theta)$, unless some cases when $G=\Sp_{2n}(q)$ with
 $q\equiv3\pmod4$.
\end{thm}

The proof, which is not too hard, uses Lusztig's Jordan decomposition and the
degree formulas.

But now $B=U.T$, with $T\le G$ a maximal torus, and $U\le\ker(\theta)$, so
in fact $\theta\in\Irr(T)$, and
$$\Ind_B^G(\theta)=\Ind_B^G(\Infl_T^B(\theta))=\RTG(\theta).$$
To check the criterion, we need to know the action of $\Aut(G)$ on the
constituents of $\RTG(\theta)$ with $T\le B$. But the decomposition of
$\RTG(\theta)$ is controlled by the Iwahori--Hecke algebra of the relative
Weyl group $W(\theta)=N_G(T,\theta)/T$:
$$\End_{\CC G}(\RTG(\theta))=\cH(W(\theta),q)\cong \CC W(\theta).$$
(As seen in the case of $\GL_n(q)$ above). This does allow us to compute
the action of $\Aut(G)$ on $\Irr_{2'}(G)$; more considerations are needed
on the local side, and extendibility has to be guaranteed.

\begin{thm}[Malle--Sp\"ath, 2016]   \label{thm:McK}
 The McKay Conjecture~\ref{conj:McK} holds for the prime $p=2$.
\end{thm}
 
In order to generalise this to other primes, one needs to work with more
general Levi subgroups: Let $\bP\le\bG$ be an $F$-stable parabolic subgroup,
and $\bL\le\bP$ an $F$-stable Levi complement, with finite groups of fixed
points $L=\bL^F\le P=\bP^F\le G$. Then there is the functor of
Harish-Chandra induction
$$\RLG:\CC L\text{-mod}\rightarrow\CC G\text{-mod},\qquad
  M\mapsto\Ind_P^G\Infl_L^P(M).$$
The special case $L=T\le P=B$ was considered above.

Now $\la\in\Irr(L)$ is called \emph{cuspidal} if it does not occur as
constituent of $\RML(\mu)$ for any proper Levi subgroup $M<L$, $\mu\in\Irr(M)$.
Again the decomposition of $\RLG(\la)$, with $\la$ cuspidal, is controlled
by the Iwahori--Hecke algebra of a relative Weyl group (Howlett--Lehrer
\cite{HL80}). To apply the above argument, one needs to understand the action
of automorphisms on cuspidal characters.

\begin{thm}[Malle (2017)]   \label{thm:cusp}
 Let $G$ be quasi-simple of Lie type. Then the action of $\Aut(G)$ on the
 cuspidal characters of $G$ lying in quasi-isolated series is known.
\end{thm}

This does, however, not yet solve the extension problem, and moreover the local
situation also needs to be studied.

\section{Blocks  and  characters of  finite  simple groups.}
As seen in the discussion around the McKay conjecture,  in order to make a
success of the reduction strategy  for the local-global conjectures one needs
very detailed knowledge both of the character theory of finite simple groups
as well as of their $p$-local structure. For the block-wise versions of these
conjectures we require  this information at a yet finer level.
A first step would be to obtain workable descriptions of block partitions of
irreducible characters and the corresponding defect groups. The block
distribution problem for finite (quasi and almost) simple groups falls
naturally into four cases: 
\begin{itemize}
\item  sporadic groups
\item  alternating groups
\item  finite groups of Lie type in describing characteristic
\item  finite groups of Lie type in non-describing characteristic
\end{itemize}
Of these the most difficult case is the last. We will discuss this case at
some length. Block distributions in sporadic groups can be worked out through
the  ATLAS character tables. The third  case, namely the blocks of  finite
groups of Lie type in defining characteristic is in some sense the easiest as
there are very few blocks (see Example~\ref{exa:defining}). In
Example~\ref{exa:Sn-local} we give a flavour of the first case by describing
the block distribution for finite symmetric groups. 

\subsection{Local Block Theory}   \label{sec:locblock}
In order to get started we need to recall some foundational results from local
block theory. As in Section~\ref{sec:lg} let $\cO\ge\ZZ_p$ be a large enough
extension. Let $k:=\cO/\fP$ be the residue field of $\cO$ and
$\bar{\ }:\cO\rightarrow k$, $\alpha\mapsto\bar\alpha:=\alpha+\fP$, the
natural epimorphism. The block decomposition
$$\cO G= B_1\oplus\ldots\oplus B_r $$
induces the unique decomposition
$$kG = \bar B_1\oplus\ldots\oplus \bar B_r$$
into a direct sum of minimal two-sided ideals of $kG$ where for any element
$a =\sum_{g \in G} \alpha_g g$ of $\cO G$ we denote $\bar a:=
\sum_{g \in G} \bar \alpha_g g$. These decompositions correspond to unique
decompositions  
$$1_{\cO G} = e_{B_1} + \ldots + e_{B_r},$$ 
$$1_{kG} = e_{ \bar B_1} + \ldots + e_{\bar B_r},$$
of $1_{\cO G}$ and of $1_{kG} $ into a sum of central primitive idempotents,
called \emph{block  idempotents} of $kG$. The expression for $ e_{\bar B}$ is
obtained by reducing coefficients modulo $p$. Thus we have  bijections
$B_i\leftrightarrow\bar B_i\leftrightarrow e_{\bar B_i}$ between the sets of
blocks of $\cO G$, blocks of  $kG$  and block idempotents of $kG$. By a defect
group of
$\bar B_i$, $e_{B_i}$ or $e_{\bar B_i}$ we mean a defect group of $B_i$.
Similarly  we may denote $\Irr(B_i)$ by $\Irr(\bar B_i)$, $\Irr(e_{B_i})$
or $\Irr(e_{\bar B_i})$.

Block idempotents can be read off the character table of $G$.
If $B$ is a block of $G$, then
$$e_B=\sum_{\chi\in\Irr(B)}\frac{\chi(1)}{|G|}\sum_{x\in G_{p'} }\chi(x)x^{-1},$$  
is the sum of central idempotents of $KG$ corresponding to the elements of
$\Irr(B)$.  

\begin{exa}   \label{exa:pconstrained}
For $G$ a finite group, $ O^p(G)$ denotes the smallest normal subgroup of $G$
with quotient a $p$-group and $O_p(G)$ denotes the largest normal $p$-subgroup
of $G$.
\par
(a) For any block $B$ of $G$, $e_B\in\cO O^p(G)$.
\par
(b) For any block $B$ of $G$ and any normal subgroup $N$ of $G$,
$e_B\in\cO C_G(N)$. In particular, if $C_G(O_p(G))\leq O_p(G)$, then the
principal block is the unique block of $G$. 
\par
(c) If $G =  G_1 \times G_2$ is a direct product then the block idempotents
of $\cO G$ are of the form $e_1e_2$ where $e_i$ is a block idempotent of
$\cO G_i$, $i=1,2$.
\end{exa}

Let $Q$ be a $p$-subgroup of $G$. For an element $a =\sum_{g\in G}\alpha_g g$
of $kG$ set
$$\Br_Q(a) = \sum_{g\in C_G(Q)}\alpha_g g\in kC_G(Q).$$
The Brauer map
$$\Br_Q : kG \to  kC_G(Q), \qquad  a \mapsto \Br_Q(a),$$ 
restricts to a multiplicative map on $Z(kG)$.

\begin{thm} [Brauer's first main theorem]   \label{thm:Br1}
 Let $D$ be a $p$-subgroup of $G$. The map 
 $$e \mapsto  \Br_D( e ) $$
 induces a bijection between block idempotents of $kG$ with defect group $D$
 and block idempotents  of $kN_G(D)$ with defect group $D$. If  $B$ is a
 $p$-block of $G$ with defect group $D$, then $ \Br_D ( e_{\bar B} )$ is the
 block idempotent of the Brauer correspondent of $B$ in $ kN_G(D)$.
\end{thm} 

A \emph{$G$-Brauer pair} (also known as \emph{subpair}) is a pair $(Q,e)$
where $ Q\leq G$ is a $p$-subgroup of $G$ and $e$ is a block idempotent of
$kC_G(Q)$. We denote by $\cP(G)$ the set of $G$-Brauer pairs and for a block
$B$ of $G$ we denote by $\cP(B)$ the subset of $\cP(G)$ consisting of Brauer
pairs $(Q,e)$ such that $\Br_Q( e_{\bar B})e\ne 0$, the elements of $\cP(B)$
are called \emph{$B$-Brauer pairs}. It is easily seen that there is a partition
$$\cP(G)={\mathcal  P}(B_1)\sqcup\ldots\sqcup\cP(B_r)$$
where $ B_1, \ldots, B_r $ are  the blocks of $G$.

We would like to relate the block decomposition of $\Irr(G)$ with the block
decomposition of $\cP(G)$. Brauer's second main theorem gives us a way of doing
this. For $x$ a $p$-element of $G$ and  $\chi\in \Irr(G)$, we let
$d^x\chi: C_G(x)\to\cO$ be the function defined by
$$d^x\chi (y)  =  \begin{cases}   \chi (xy)& \text{if   $y \in G_{p'} $, }\\
                            0& \text{if  $y\notin G_{p'}$.}\end{cases}$$

\begin{thm}[Brauer's second main theorem]   \label{thm:Br2}
 Let $ B$ be a  block   of $ G$,  $x \in G $  a $p$-element and $C$ a block
 of $C_G(x)$. Suppose that there exists $\chi \in\Irr(B) $, $ \psi \in\Irr(C)$
 such that
 $$\langle d^x\chi,\psi\rangle:= \sum_{y\in C_G(s)_{p'} }\chi(xy)\psi(y)\ne0.$$
 
 Then $(\langle x\rangle,e_C)\in\cP(B)$. In particular, if $d^x\chi\ne 0$,
 then $\cP(B)$ contains an element of the form $(\langle x\rangle,e)$.
\end{thm}

The assignment $\chi\mapsto d^x\chi$ extends by linearity to a map $d^x$
from the set of $\cO$-valued class functions on $G$ to the set of $\cO$-valued
class functions on $C_G(x)$. The map $d^x$ is called the \emph{generalised
decomposition map} with respect to $x$. When we want to emphasise  the 
underlying group $G$, the generalised decomposition map is denoted $d^{x,G}$.

The set $\cP(B)$ has a nice description when $B$ is the principal block.

\begin{thm}[Brauer's third main theorem]   \label{thm:Br3}
 Let $B_0$ be the principal block of $G$ and let $(Q,e)\in\cP(G)$. Then
 $(Q,e)\in\cP(B_0)$ if and only if $e$ is the idempotent of the principal
 block of $C_G(Q)$. 
\end{thm}

The set $\cP(G)$ is a $G$-set  via 
$$^x(Q,e) = (\,^x Q,\,^ x e),\qquad\text{for all $x\in G$, $(Q,e)\in\cP(G)$}$$
where $^x a := xax^{-1}$ for $x\in G$, $a\in kG$. The subset $\cP(B)$ is
$G$-invariant for $B$ a block of $G$. In \cite{AB79} Alperin and Brou\'e
endowed $\cP(G)$ and  $\cP(B)$   with  a  $G$-poset  structure.   Let $(Q, e), (R, f) \in\cP(G)$.     We say that    $(Q,e) $ is \emph{normal in $(R,f)$}  and write   $ (Q, e) \unlhd  (R, f) $   if  $Q \unlhd R $,  $ \, ^x(Q, e) =  (Q,e)  $ for all $x \in R$ and $ \Br_R(e) f \ne 0 $.  We say that  $ (Q, e) \leq    (R, f) $   if there exists a chain  of   normal inclusions 
$$  (Q, e)=: (Q_0, e_0 ) \unlhd  \ldots  \unlhd (Q_n, e_n) := (R, f) $$ 
in ${\cP }(G) $ starting at $(Q, e)$ and ending at $ (R, f) $.

\begin{thm} [Alperin--Brou\'e (1979)]   \label{thm:AlpBro}
 $(\cP(G),\leq)$ is a $G$-poset. For any $(R,f)\in\cP(G)$  and any $Q \leq R $,
 there exists a unique block $e$ of $kC_G(Q)$ such that  $(Q,e)\leq(R, f)$.
 The sets $\cP(B)$ as $B$ runs over the blocks of $G$ are the connected
 components of $(\cP (G),\leq)$. For a block $B$ of $G$, 
 \begin{enumerate}  [\rm(a)]
  \item $\cP(B)$ is $G$-invariant and $(1, e_{\bar B})$ is the unique minimal
   element  of ${\cP(B)} $.
  \item $G$ acts transitively on the set of maximal elements  of  $\cP(B) $
   and an element  $(D,d) $ of  $\cP(B) $   is maximal  if and only if $D$ is a
   defect group of $B$.
 \end{enumerate}
\end{thm} 

The  following is sometimes known as  Brauer's extended first main theorem.

\begin{thm} [Recognition of maximal Brauer pairs]  \label{thm:recog}
 Let $ (Q, e) \in\cP(G)$. Then $(Q,e)$ is maximal if and only if there exists
 $\theta \in \Irr(e) $ such that  
 \begin{itemize}
  \item $ Z(Q) \leq \ker (\theta) $,
  \item as a character of $C_G(Q)/Z(Q) $, $\theta  $ is of  defect $0$, and
  \item $ N_G(Q, e)/ Q C_G(Q) $  is a $p'$-group.
 \end{itemize} 
\end{thm} 

\begin{exa}   \label{exa:Sn-local}
As discussed in Example~\ref{exa:first}, an irreducible character $\chi^\la$
of $\fS_n$ lies in a $p$-block of defect zero if and only if $\lambda $ is a
$p$-core. By the Nakayama  conjecture, posed in  \cite{Na} and proved by
Brauer and Robinson \cite{BrRo}, given partitions $\la$, $\la'$ of $n$, the
corresponding characters $\chi^\la$ and $\chi^{\la'}$ lie in the same
$p$-block of $\fS_n$ if and only if $\la $ and $\la'$  have the same $p$-core.

Puig \cite{Pu86} showed how the  block distribution of irreducible characters
matches up with the block distribution of $\cP(\fS_n)$.    
Let $Q$ be a $p$-subgroup  of $\fS_n$. Then $n =  m+ pw $, where $Q$  fixes $m$
points  and moves $pw$ points in the natural permutation representation of $\fS_n$,
and
 $$C_{\fS_n}(Q)  =  \fS_{m} \times  C_{\fS_{pw}}(Q), \qquad
  N_{\fS_n}(Q)  = \fS_{m}  \times  N_{\fS_{pw}}(Q).$$
The  action of $Q\leq\fS_{pw}$ is fixed-point free and it is not hard to show
that this implies  that 
$$C_{C_{\fS_{pw}}(Q)}(O_p(C_{\fS_{pw}}(Q))) \leq O_p(C_{\fS_{pw}}(Q)).$$
By Example~\ref{exa:pconstrained}(b), the principal block  is the unique block
of $C_{\fS_{{pw}}}(Q)$. Thus by Example~\ref{exa:pconstrained}(c) every Brauer
pair with first component $Q$ is of the form $ ef $ where $e$ is a block
idempotent of $k\fS_m$ and $f$ is the principal block idempotent (in fact the
identity element) of $k C_{\fS_{pw}}(Q)$.   

Next, we describe the inclusion of Brauer pairs. This is a difficult and subtle
step and is carried out inductively --- a crucial ingredient is the
Murnaghan--Nakayama rule which is an inductive  combinatorial rule for
calculating  values of irreducible characters of symmetric groups. 
Let $ Q'$ be a $p$-subgroup of $\fS_n$ containing $Q$ and suppose that $Q'$
moves $pw'$ points and fixes $m'$ points. Since $ Q\leq Q'$, $w' \geq  w$ and
$m' \leq m $. Let $e' f'$ be a block of $C_{\fS_n} (Q')$ with $e'$ a block of
$\fS_{m'}$ and $f'$ the principal block idempotent of 
$k C_{\fS_{pw'}}(Q')$. Then one can show that $(Q,ef)\leq(Q,e' f')$ if and
only if $\la $ is obtained from $\mu$ by adding a sequence of $p$-hooks for
$\chi^\lambda\in\Irr(e)$, $\chi^\mu\in\Irr(e')$. 

Finally, we describe the maximal pairs. Applying Theorem~\ref{thm:recog} in
both directions, one sees that $(Q,e f_Q)$ is a maximal $G$-Brauer pair if and
only if $(1,e)$ is a maximal Brauer pair for $\fS_m$ and $(Q,f)$ is a maximal
$\fS_{pw}$-Brauer pair. By the inclusion rule described above, $(1,e)$ is a
maximal Brauer pair for $\fS_m$ if and only if $e$ is the block idempotent of
a block of $\fS_m$ of defect zero, that is, a block whose unique irreducible
character is of the form $\chi^\la$ where $\la$ is a $p$-core. Since $f$ is a
principal block idempotent, by Brauer's third main theorem $(Q,f)$ is a maximal
$\fS_{pw}$-Brauer pair if and only if $Q$ is a Sylow $p$-subgroup of $\fS_{pw}$. 
Thus, the $G$-conjugacy classes of maximal $\fS_n$-Brauer pairs are in
bijection with pairs $(\mu, w)$ where $w$ is a non-negative integer such that
$ pw \leq n$ and $\mu  $ is a partition of  $n-pw$ which is a $p$-core.
By Theorem~\ref{thm:AlpBro} the $G$-conjugacy classes of maximal $G$-Brauer
pairs are in one-to-one correspondence  with the blocks of $G$. Hence we obtain
a bijection  
between the set of blocks of $\fS_n $ and pairs as  above; the
character $\chi^\la\in\Irr(\fS_n)$ lies in the block indexed by the pair
$(\mu,w)$ if and only if $\mu$ is the $p$-core of $\la$. If a block $B$ is
indexed by the pair $(\mu, w)$, then $w$ is called the \emph{weight of $B$}
and $\mu $ is called the \emph{core of $B$}.

If $B$ has weight $w$ and core $\mu $, then a Sylow $p$-subgroup $P$ of
$\fS_{pw}$ is a defect group of $B$ and the Brauer correspondent of $B$ in 
$$N_{\fS_n}(P)=\fS_{m}\times N_{\fS_{pw}}(P)\qquad\text{where $\mu\vdash m$}$$     
has the form
$$B_\mu  B_w$$
where $B_\mu$ is the block of $\fS_m$ indexed by the pair $(\mu,0)$
and $B_w$ is the principal block of $N_{\fS_{pw}}(P)$. The irreducible
characters in the Brauer correspondent are of the form 
$$ \chi^{\mu}. \eta, \qquad \text{where $ \eta \in \Irr(C_w) $}. $$   
From this, it is easy to check that the map $ \chi^\mu. \eta \mapsto\eta$
is a height preserving bijection between the set of irreducible characters of
the Brauer correspondent of $B$ and the set of irreducible characters of $C_w$.
Thus we obtain:
\begin {enumerate}[\rm(I)]
\item  Given any non-negative integer $w$, there is a height preserving
 bijection between the irreducible characters of a Brauer correspondent of a
 weight $w$-block of a symmetric group and the principal block of
 $N_{\fS_{pw}}(P)$, where $P\leq\fS_{pw}$ is a Sylow  $p$-subgroup. In
 particular, there is a height preserving bijection between the irreducible
 characters of the Brauer correspondents of any two weight $w$ blocks of
 (possibly different) symmetric groups.  
\end{enumerate}
In \cite{E90}  Enguehard showed that the global analogue of the above
statement also holds, namely:
\begin{enumerate}[\rm(II)]
\item Given any non-negative integer $w$, there is a height preserving
 bijection between the irreducible characters of any two weight $w$ blocks of
 (possibly different) symmetric groups.  
\end{enumerate}

Thus the problem of checking a desired local-global statement for blocks of
symmetric groups can often be reduced to checking it for a single block of any
given weight~$w$.  

Let us consider Brauer's height zero conjecture (Conjecture~\ref{conj:BHZ})
for $p=2$. Since blocks with the
same weight have isomorphic defect groups, in order to prove the height zero
conjecture for blocks of symmetric groups it suffices to prove that it holds
for the principal block $B$ of $\fS_{2w}$.
The defect groups of $B$ are the Sylow $2$-subgroups of $S_{2w}$. Hence $B$
has abelian defect groups if and only if $w=1$. On the other hand,   
$$\Irr(B)=\{\chi^{\la}\mid\la\vdash 2w\text{ and $\la$ has empty $2$-core}\}.$$
Since $B$ is the principal block, $\Irr_0(B) = \Irr(B)\cap\Irr_{2'}(\fS_{2w})$.
Thus we are reduced to checking the following statement:
$$\text{$w\geq 2$ if and only if there is $\la \vdash n$ such that $\la$
  has empty $2$-core and $2$ divides  $\chi^{\lambda}(1)$.}$$
The backward implication is  immediate as the only partitions of $2$ are $(2)$
and $(1,1)$. Now suppose that $ w \geq 2$. Then $ \la = (2w-1,1)$ has empty
$2$-core --- we first remove successively $w-1$ horizontal hooks of length $2$
from the first part of the Young diagram,  then remove the remaining  vertical
$2$-hook. The hook length formula (see Example~\ref{exa:Sn}) easily yields that
$\chi^\lambda$ has even degree.

The local-local and global-global bijections described in~(I) and~(II) above
are shadows of deeper categorical equivalences. It is quite easy to deduce
from the above discussion that any two Brauer correspondents of blocks of
symmetric groups with the same weight are Morita equivalent. Much harder is
the analogous global version proven by Chuang and Rouquier \cite{CR}: any two
$p$-blocks of symmetric groups of the same weight are derived equivalent.
\end{exa}

\begin{exa} \label{exa:defining}
Groups of Lie type in characteristic $p$ have very few $p$-blocks. The main
structural reason for this is the Borel--Tits theorem. As in
Section~\ref{subsec:lietypegroups}, let $\bG$ be a simple algebraic group over
$\overline{\FF}_p$ with a Frobenius endomorphism $F:\bG\rightarrow\bG$ and let
$ G=\bG^F$. Suppose that $\bG$ is simply connected and $Z(G)=1$. Then the
Borel--Tits theorem \cite[Thm.~26.5]{MT} implies that if $Q$ is a non-trivial
$p$-subgroup of $G$, then 
$$C_{N_G(Q)}(O_p(N_G(Q)))\leq O_p(N_G(Q)).$$
By Example~\ref{exa:pconstrained}(b) applied to $N_G(Q)$, the principal block    
is the unique block of $N_G(Q)$. Put another way, the identity element is the
unique central idempotent of $kN_G(Q)$. Now suppose that $f$ is a block
idempotent of $kC_G(Q)$. Then it is easy to see that the sum $f'$ of the
distinct $N_G(Q)$-conjugates of $f$ is a central idempotent of $kN_G(Q)$. Hence
$f'$ is the identity element of $kN_G(Q)$. From this it follows that $f=f'$ is
the identity element of $kC_G(Q)$, and consequently $f$ is the principal block
of $kC_G(Q)$. In other words, the only $G$-Brauer pair with first component
$Q$ is the pair $(Q,f)$, where $f$ is the principal block idempotent of
$kC_G(Q)$. Now Brauer's third main theorem (Theorem~\ref{thm:Br3}) gives that
if $(Q, e)\in\cP(G)$ with $Q\ne 1$, then $(Q,e)\in\cP(B_0)$, where $B_0$ is
the principal block of $G$.  We conclude that the irreducible characters
$\chi$ of $G$ lying outside the principal block are all of defect zero. It
turns out that there is only one character of defect zero, namely the the
Steinberg character \cite[Thm.~8.3]{Hu06}. Thus $G$ has precisely two blocks:
the principal block and the block containing the Steinberg character.   

If the assumption that $Z(G)=1 $ is dropped, then we obtain more blocks, but
the extra blocks are in bijection with the non-trivial elements of $Z(G)$.
More precisely, we have the following \cite[Thm.~8.3]{Hu06}: Suppose that
$\bG$ is simple and simply connected. The blocks of non-zero defect of $G$ are
in bijection with the elements of $Z(G)$ and all have the Sylow $p$-subgroups
of $G$ as defect groups. There is exactly one block of zero defect, namely
the block containing the Steinberg character of $G$. 
\end{exa}

\subsection{Blocks of groups of Lie type in non-defining characteristic.}
We continue in the setting and notation of Section~\ref{subsec:lietypegroups},
so $\bG$ is a connected reductive group over $\overline{\FF}_p$ with a
Frobenius endomorphism $F:\bG\rightarrow\bG$ and $G=\bG^F$.
Let $\ell$ be a prime different from $p$. Our aim is to  give a broad idea of
how  the $\ell$-block partition of $\Irr(G)$ can be described in terms of
Lusztig's parametrization of $\Irr(\bG^F)$. For notational simplicity for any
$F$-stable subgroup $\bH$ of $\bG$ or of  the dual group $\bG^*$, we will
denote by $H$ the $F$-fixed point subgroup $\bH^F$. For a character $\chi$ of
$G$ and $x\in G$ an $\ell$-element denote by $d^{x,G}\chi : C_G (x) \to \cO$
the function $d^{x}\chi$ as defined for Brauer's second main theorem (see
Section~\ref{sec:locblock}).

A key starting point is the following result which relates generalised
decomposition maps to Deligne--Lusztig induction and restriction. Let $\bT$ be
an $F$-stable maximal torus of $\bG$, $T:=\bT^F$, and let $x\in T_\ell$.
Set $\bH:=C_\bG^\circ(x)$, the connected component of the centraliser of $x$
in $\bG $. The group $\bH$ is again a connected reductive group which is
$F$-stable  and $H:=\bH^F$ is a normal subgroup of $C_G(x)$ which may be proper
(equality holds for example if $C_\bG(x)$ is itself connected). However,
since $x$ is an $\ell$-element, the general structure theory of connected
reductive groups gives that the index of $H$ in $C_G(x)$ is a power of $\ell$.
By definition, for any character $\chi$ of $G$, $d^{x,G}\chi$ is a function
which takes zero values on $\ell$-singular elements. Since $H$ contains all
$\ell$-regular elements of $C_G(x)$, we may regard $d^{x,G}\chi$ as a
function from $H$ to $\cO$. 
Note that $\bT$ is a maximal torus of $\bH$ so $\RTH$ and $\sRTH$ are defined. 

\begin{thm}   \label{thm:commut}
 $$\sRTH(d^{x,G}\chi)=d^{x,T}(\sRTG(\chi))\qquad
   \text{for all $\chi\in\Irr(G) $}, $$
 that is, Deligne--Lusztig restriction commutes with generalised decomposition
 maps.
\end{thm}

For an $F$-stable maximal torus $\bT$ of $\bG$ let $\Irr(T)_{\ell'}$ denote the
subset of irreducible characters of $T$ of $\ell'$-order, that is
$\Irr(T)_{\ell'}$ consists of those irreducible characters $\theta$ such that
$T_\ell\leq \ker(\theta)$. Let $\cE(G,\ell')$ denote the subset of $\Irr(G)$
consisting of those $\chi$ such that $\langle\RTG(\theta),\chi\rangle\ne 0$
for some $F$-stable maximal torus $\bT$ of $\bG$ and some
$\theta\in\Irr(T)_{\ell'}$. The following theorem illustrates how Brauer's
local block theory and the theory of Deligne--Lusztig characters come together.

\begin{thm}   \label{thm:BrLu}
 Let $\bT$ be an $F$-stable maximal torus of $\bG$, $\theta\in\Irr(T)_{\ell'}$.
 Suppose that     
 $$C_\bG(T_\ell) =\bT.\eqno{(*)} $$   
 Then all elements of $\cE (G|(\bT, \theta))$ lie in  the same $\ell$-block
 $B$ of $G$. Further, if $e$ is the block idempotent of $kT =  k C_G(T_\ell)$
 containing $\theta$, then $ (T_\ell, e ) $ is a $B$-Brauer pair.  
\end{thm} 
 
The proof of the above  theorem  goes the following way (details may be found
in  \cite[Props.~2.12,~2.13,~2.16]{KM13}). By general structure theory the
group $C_\bG(Q)$ is a reductive algebraic group for any $Q \leq T_\ell$. For
simplicity we will assume that $C_\bG (Q)$ is also connected. 
\begin{enumerate}
\item[Step 1:] Let $\chi\in\cE(G|(\bT,\theta))$. By the adjointness of
 Deligne--Lusztig induction and restriction, $\theta$ is a constituent of the
 virtual character $\sRTG(\chi)$. Another key property of these maps is that
 since $\chi$ is a constituent of $\RTG(\theta)$ and $\theta\in\Irr(T){\ell'}$,
 all irreducible constituents of $\sRTG(\chi)$ belong to $\Irr(T)_{\ell'}$.
 Write
 $$\sRTG(\chi) = a_{\theta}\theta + \sum_{\tau\in\Irr(T)_{\ell'}\setminus \{\theta \}} a_{\tau} \tau, \qquad  \text{with   $a_{\theta}, a_{\tau}  \in \ZZ$,   and   $a_{\theta}  \ne 0$}. $$  
Let $ x \in T_\ell $.  Since $T$ is  an abelian group, it follows  easily  from  the definition of  generalised decomposition maps    that  if $\theta_1, \theta_2  \in \Irr(T)_{\ell'} $  then 
$$ \langle  d^{x, T} \theta_1, \theta_2 \rangle   = \frac{1}{|T_\ell|} \langle  \theta_1, \theta_2 \rangle. $$
Applying this to  the  above expression for $\sRTG(\chi)$  gives
$$ \langle d^{x, T} ( \sRTG (\chi)),   \theta   \rangle   = \frac{1}{|T_\ell|} a_{\theta} \ne 0 . $$
By the commutation property in Theorem~\ref{thm:commut},
$$ \langle d^{x, T} ( \sRTG (\chi)),   \theta   \rangle =  \langle \sRTH ( d^{x, G} \chi ),   \theta   \rangle =\langle d^{x, G} \chi,   \RTH (\theta)   \rangle $$ where $ \bH=C_\bG(x) $ and where the second equality holds by adjointness.
Now  Brauer's Second  Main Theorem~\ref{thm:Br2} implies that there exists a
$B$-Brauer pair $(\langle x \rangle , f) $ such that 
$$ \cE (C_G(x)   | (\bT, \theta)) \cap \Irr (f) \ne 0 . $$   

\item[Step 2:]  Let  $ \{ x_1,  \ldots, x_m \} $   be a generating set  of  $T_\ell$  and let $ Q_i =\langle  x_1, \ldots, x_i \rangle $, $ 1\leq  i \leq m $.   Applying  Step 1  repeatedly  with $G$  replaced   by $  C_G(Q_i) $, we  obtain a   sequence of inclusions of   $B$-Brauer pairs
$$ (Q_1, f_1)  \leq   \ldots \leq  (Q_m,  f_m )  $$   such that  
$$ \cE (C_G(Q_i)   | (\bT, \theta)) \cap \Irr (f_i) \ne 0  \qquad\text{for all $i$}. $$ 
 Since $Q_m =   T_\ell $, the hypothesis (*) implies that   $ C_G(Q_m) = T $. Since $\theta $ is in $\Irr(e) $ as well as  in  $\Irr(f_m)$, $f_m = e $, and
the uniqueness of inclusion of Brauer pairs allows us to conclude.  
\end{enumerate}

\begin{exa} \label{exa:GLBrLu}
(a) Let $\bG =\GL_n$, $G=\GL_n(q)$. If $\ell$ divides $q-1$, then the 
$F$-stable maximal torus $\bT$ of diagonal matrices of $\bG$ satisfies the
hypothesis~(*) of Theorem~\ref{thm:BrLu}. 
Thus, for any $\theta \in \Irr(T)_{\ell'}$,  all constituents of $ \RTG (\theta) $ lie in the same $\ell$-block.     By contrast, if $\ell$ does not divide $q-1$, then    $\bT^F$    has trivial $\ell$-part. \par
(b) Suppose that $\bG$ is simple of classical type $A$, $B$, $C$, or $D$ and
$\ell = 2$. For any $F$-stable maximal torus  $\bT $ of $ \bG$, all elements
of $\cE(G|(\bT,1))$ lie in the principal 2-block of $G$ \cite{CE93}.
The key property is that $T_2$ is non-trivial for all $F$-stable maximal tori
$\bT$  (of all $F$-stable Levi subgroups) of $\bG$. One  applies Step~(1) of
the proof of Theorem~\ref{thm:BrLu}  to some  non-trivial   $x$ in $T_2 $  and then proceeds by induction on the dimension (as algebraic group) of $C_\bG(x)$. 
\end{exa}

The hypothesis (*) of Theorem~\ref{thm:BrLu} does not hold often enough to
obtain satisfactory control of  block distribution of characters. The strategy
to get around this is to replace $F$-stable maximal tori by a certain class of
well behaved $F$-stable Levi subgroups.  

Let $ \bP \leq \bG $ be a parabolic subgroup and let $ \bL \leq  \bP$ be a Levi complement.  If $\bL$    is $F$-stable, then we  have a pair of  
mutually adjoint linear maps, called Lusztig induction and restriction,
$$\RLG:\ZZ\Irr(\bL^F)\to\ZZ\Irr(\bG^F),\qquad
  \sRLG:\ZZ\Irr(\bG^F)\to\ZZ  \Irr(\bL^F), $$
enjoying many of the same properties as  the maps $\RTG $ and $\sRTG$.  The
construction involves the parabolic subgroup $\bP$,  and hence strictly speaking  the notation for  $\RLG$ should include $\bP$. However, we take the liberty of omitting this as in almost all situations it is known that  the construction is independent of the choice of $\bP$.   If $\bP$ is also $F$-stable  then Lusztig  induction  of an irreducible character $\chi$   of $L$  is   the same as Harish-Chandra induction    of $\chi $  as considered in the previous section.

For $\la\in\Irr(L)$ we let $\cE( G|(\bL,\la))$ denote the set of irreducible
constituents of   $\RLG (\la) $.       We have an analogue of Theorem \ref{thm:BrLu}  which we state under an assumption on the prime $\ell $ being ``large enough".  This   assumption 
can be replaced by other conditions, e.g. $|\cE ( L, \ell' )\cap\Irr (e)|=1$.

\begin{thm}   \label{thm:BrLuLevi}
 Suppose that $\ell\geq 7$. Let $\bL$ be an $F$-stable Levi subgroup of $\bG$
 and let  $\la \in \cE (L, \ell' )$.  Suppose that
 $$ C_\bG (Z(L)_\ell )=\bL.\eqno{(**)} $$      
 Then  all elements of $\cE (G  | (\bL, \la)) $ lie in  the same $\ell$-block
 $B$ of $G$. Further, if $e$ is the block idempotent of $kL$ containing
 $\theta$, then $(Z(L)_\ell,e)$ is a $B$-Brauer pair.   
\end{thm}

The advantage of Theorem~\ref{thm:BrLuLevi} over Theorem~\ref{thm:BrLu} is that
there are many  Levi subgroups satisfying Condition~(**)-the disadvantage is that $ \RLG$ and $ \sRLG $ are   harder to work with than $ \RTG$ and $ \sRTG $.   It  is  known that every Levi subgroup of $\bG$   is   of the form $\bL=C_\bG(\bS)$ where $\bS \leq \bG $ is a (not
necessarily maximal) torus  $S$. Clearly,  if $\bS$ is  $F$-stable    then so is $\bL$. The class of Levi subgroups that  is  well adapted to     Condition (**) are centralisers of particular $F$-stable tori  which we now describe.

To every  $F$-stable torus  $\bS$ of $\bG$ is associated  a  monic polynomial  $P_{\bS}(x) $ with integer coefficients  called the \emph{polynomial order of   $\bS$}   such that $$ |\bS^{F^m} | =  P_{\bS} (q^m)\qquad \text{for infinitely many integers $m$.}$$ 
The polynomial order of  $\bS $ is uniquely   defined  and  is a   product  of cyclotomic polynomials  $ \Phi_d(x) $, $ d \in {\mathbb N}$.  If the polynomial order of $\bS$ is a  power of  $\Phi_d (x)  $  for  a  single integer $d$, then we say that   $\bS $ is a 
\emph{$\Phi_d $-torus}.  If $\bL$ is the centraliser in $\bG$ of a $\Phi_d $-torus, then $\bL$ is  said to be a \emph{$d$-split Levi subgroup of $\bG$}.      

The following    theorem   of Cabanes and Enguehard which we state under some
simplifying hypotheses shows that  the class of $d$-split Levi subgroups (for
a particular $d$) satisfies (**).

\begin{thm} [Cabanes-Enguehard (1999)]  \label{thm:BrLucond}
 Suppose that $Z(\bG)$ is connected, $[\bG,\bG]$ is simply connected and
 $\ell\geq 7$. Let $d= d_{\ell}(q)  $ be the  order of $q$ modulo  $\ell$. Then every  $d$-split Levi subgroup of
 $\bG$ satisfies condition (**) of Theorem~\ref{thm:BrLuLevi}.  
\end{thm}

\begin{exa}
Let $\bG =\GL_n(\overline{\FF}_q)$, $G =\GL_n(\FF_q)$. If $\bL$ is an $F$-stable
Levi subgroup of $\bG$, then 
$$ L \cong  \GL_{a_1}(q^{m_1})  \times  \cdots \times \GL_{a_r}(q^{m_r})$$
for some positive integers $a_i $ and $m_i$, $1\leq i \leq r$, such that
$\sum_i a_i m_i =n$. The group $\bL$ is $d$-split if and only if $m_i = d $
for all $1\leq i \leq r $.
\end{exa}

\subsection{Lusztig series and Bonnaf\'e--Rouquier reduction.}   \label{subsec:luboro}
Another key feature of block theory in non-defining characteristic is that the
subset $\cE(G,\ell')$ controls the $\ell$-block distribution of irreducible
characters. This is made precise in the  following theorem. Note that
$\cE(G,\ell')$ is the union of the Lusztig series $\cE(G,s)$ as $s$ runs over
conjugacy classes of semisimple elements of $\ell'$-order in the dual group
$G^*$.

\begin{thm}[Hiss (1989), Brou\'e-Michel (1988)]   \label{thm:broue-michel-hiss}
 Let $B$ be an $\ell$-block of $G$. There exists a semisimple $\ell'$-element
 $s$ of $G^*$, unique up to conjugacy in $G^*$ such that 
 $$ \Irr(B)\cap\cE(G,s)\ne\emptyset.$$
 If $t$ is a semisimple element of $G^*$ such that
 $\cE(G,t)\cap\Irr(B)\ne\emptyset$, then $t_{\ell'}$ is conjugate in $G^*$ to
 $s$.
\end{thm}

For a semisimple $\ell'$-element $s$ of $G^*$, let $\cE_\ell(G, s)$ be the
union of Lusztig series $\cE(G,t)$ where $t$ runs over all semisimple elements
of $G^*$ whose $\ell'$-part is $G^*$-conjugate to $s$. The above theorem
implies that $\cE_\ell(G,s)$ is a union of $\ell$-blocks of $G$. Thus, the
$\ell$-block  distribution problem can be broken down as follows. For each
(conjugacy class of)   semi-simple  $\ell'$-element   $s$ of $G^* $ describe:
 \begin{enumerate}[\rm(I)]
  \item  The $\ell$-block distribution of $\cE(\bG,s)$.
  \item  For each non-trivial semisimple $\ell$-element $t$ in $C_{\bG^*}(s)$
   describe the $\ell$-block distribution of $\cE(\bG,st)$.
\end{enumerate}
This approach is compatible with Theorem~\ref{thm:BrLuLevi} since if $s$ is a
semisimple $\ell'$-element of $L^*$ for some $F$-stable Levi subgroup $\bL $ of
$\bG$, and $\la\in\cE(L,s)$, then all elements of $\cE(G|(\bL, \la))$ lie in
$\cE(G,s)$.

The following powerful theorem of Bonnaf\'e and Rouquier \cite{BR03} allows
for a dramatic shrinking of the magnitude of the problem.

\begin{thm}[Bonnaf\'e--Rouquier (2003)]   \label{thm:BR}
 Let $s\in G^*$ be semisimple such that $C_{\bG^*}(s)\le\bL^*$ for some
 $F$-stable Levi subgroup $\bL^*$ of $\bG^*$. Then the product of $\ell$-block
 algebras in $\cE_\ell(G,s)$ is Morita equivalent to the product of
 $\ell$-block algebras in $\cE(L,s)$.
\end{thm}

So, inductively, we only need to study the  blocks in Lusztig series
$\cE_\ell(G,s)$ such that $s$ is {\it quasi-isolated} in $\bG$, that is, such
that $C_{\bG^*}(s)$ is not contained in any proper Levi subgroup of
$\bG$. The blocks in non-quasi-isolated series can be "recovered" from
blocks of groups of Lie type where the underlying  algebraic group is of
smaller dimension than that of $\bG$. There is a price to be paid here: since
Levi subgroups of a simple algebraic group are not simple, while working in
the inductive set-up we cannot restrict ourselves to only considering simple
algebraic groups. We need to take into account all Levi subgroups as well.

Recently, Bonnaf\'e, Dat and Rouquier \cite{BDR} have given an improvement of
Theorem~\ref{thm:BR} which in most situations reduces the set of semisimple
elements that need to be considered even further, namely to \emph{isolated}
elements. These are elements $s$ such that the connected component of
$C_{\bG^*}(s)$ is not contained in a proper Levi subgroup of $\bG$.

\subsection{Unipotent blocks  and $d$-Harish-Chandra  theory.}
The best understood class of blocks are the \emph{unipotent} blocks. These are
the blocks in $\cE_\ell(G,1)$. By Theorem~\ref{thm:broue-michel-hiss}, the
unipotent blocks are precisely the blocks which contain a unipotent character.
In \cite{BMM}, Brou\'e, Malle and Michel generalised the Harish-Chandra
theory of Howlett and Lehrer to the context of $d$-split Levi subgroups (see
the discussion before Theorem~\ref{thm:cusp}). This $d$-Harish-Chandra theory
is an important ingredient in the solution of the block distribution problem.

For $d$ a positive integer, let $\cU_d(G)$ denote the set of all pairs
$(\bL,\la)$ such that $\bL$ is a $d$-split Levi subgroup of $\bG$ and $\la$
is an irreducible unipotent character of $L$. We regard $\bG$ as a $d$-split
Levi subgroup of itself, so $(\bG,\chi)\in\cU_d(G)$ for any irreducible
unipotent character $\chi$ of $G$. The set $\cU_d(G)$ is a $G$-set via
$$\,^g(\bL,\la)=(\,^g\bL,\,^g\la),\quad
  \text{for $g\in G,(\bL,\la)\in \cU_d (G)$}. $$
There is  also  an inclusion relation  on $\cU_d(G)$  which is defined as follows: 
$$(\bL,\la)\leq(\bM,\mu)\quad\text{if $\bL\leq\bM$ and $\la$ is a constituent
  of $\RLM (\mu)  $}. $$
The pair $(\bM, \mu )$ is  said to be a \emph{unipotent $d$-cuspidal pair of
$G$} if there does not exist a  unipotent  $d$-cuspidal  pair  $(\bL,\la) \leq(\bM, \mu)$ with $\bL$
proper in $\bM$. 

For $(\bL,\la)\in\cU_d( G)$ and $\bM$ an $F$-stable Levi subgroup of $\bG$
containing $\bL$, we denote by $N_{M}(\bL, \la) $ the stabiliser in $M$ of
the pair $(\bL,\la)$ and we denote by $W_M(\bL,\la)$ the group
$N_M(\bL,\la)/L$, the \emph{relative Weyl group of $(\bL, \la)$}.

\begin{thm}[Brou\'e--Malle--Michel (1993)]   \label{thm:bmm}  Let $d$ be a positive integer.
 \begin{enumerate}[\rm(a)]
  \item  $(\cU_d(G),\leq)$ is a $G$-poset. The connected components of
   $(\cU_d(G),\leq)$ are precisely the sets $\cE(G|(\bL,\la))$ as $(\bL,\la)$
   runs  over  a  set of representatives  of  $G$-conjugacy classes of
   $d$-cuspidal  pairs of  $G$.
  \item Let $(\bL,\la)$ be  a $d$-cuspidal pair of $G$  and let $ \bM $ be a
   $d$-split Levi subgroup of $\bG$ containing $\bL$. There exists an
   isometry
   $$\ZZ \cE(M|(\bL,\la))\cong\ZZ\Irr(W_{\bM} (\bL,\la))$$
   intertwining  $\RMG $ with $\Ind_{W_M(\bL,\la)}^{W_G(\bL,\la)}$.
 \end{enumerate}
\end{thm}

The proof of Theorem~\ref{thm:bmm} is on a case by case basis and relies
heavily on
the combinatorics associated to unipotent characters. An especially delicate
point is the transitivity of $\leq$ since Lusztig induction sends characters
to virtual characters. A by-product of part~(b) of the theorem is an explicit
description of $d$-cuspidal pairs $(\bL,\la)$ and the sets $\cE(\bG|(\bL,\la))$.
For classical  groups, this  description is in terms of the combinatorial
``yoga" associated to partitions and  symbols labelling unipotent characters and is 
in terms of tables for exceptional groups.

The set $\cE( G|(\bL,\la))$ for a given $d$-cuspidal pair $(\bL,\la)$ is
called the \emph{ $d$-Harish-Chandra series above $(\bL, \la)$}.
Thus, for any $d\ge1$, Theorem~\ref{thm:bmm} provides a partition  of the set
of unipotent characters   into $d$-Harish Chandra  series.
It turns out that when $ d$  is the order of $q$ modulo $\ell$  and provided that   $\ell$ is sufficiently large  the  partition   into  $d$-Harish-Chandra series coincides  with  the   block partition of $\cE(\bG,  1) $, and is also closely linked with the  $\ell$-block partition of  $\mathcal{P}(G)$.
The following theorem, which  makes this  more precise, was  proved  by   Cabanes and Enguehard \cite{CE94}.  For   very large $\ell$  it is due to  Brou\'e, Malle and Michel \cite{BMM}. The first assertion  is  covered  by  Theorem ~\ref{thm:BrLuLevi}  and Theorem~\ref{thm:BrLucond}.

\begin{thm}[Brou\'e--Malle--Michel (1993), Cabanes--Enguehard (1994)]  \label{thm:blocks}
 Suppose  that  $\ell \geq 7 $ and let $ d =d_{\ell}(q) $.
\begin{enumerate} [\rm (a) ]\item   For any   unipotent   $d$-cuspidal  pair $ (\bL, \la )  $ of $\bG$ there exists a  unique $\ell$-block   $B_{G} (\bL, \la) $   of $G$ containing   all elements   of   $\cE  ( G |  (\bL, \la ) )$.
\item  The map  $ (\bL, \la ) \mapsto  B_{G} (\bL, \la) $  induces a bijection between the $G$-classes of     unipotent  $d$-cuspidal  pairs  and  the set of unipotent blocks   of $G$.
\item   $  \Irr(B_G(\bL, \la)  )\cap \cE (\bG, 1) =\cE (G | (\bL, \la)   $  for all  unipotent   $d$-cuspidal  pairs $ (\bL, \la )  $ of $\bG$.
\item  Suppose that $[\bG, \bG]$ is simply connected.  The map $(\bM, \mu) \mapsto  (Z(M)_\ell, e (\mu) )$  is an order reversing  isomorphism from $(\cU_d (G),  \leq   ) $ onto a subset  of $(\mathcal{P} ( G), \leq ) $ where $e(\mu)$   denotes the block idempotent of $kC_G(Z(M)_{\ell} )$ associated to $\mu$.
\item  There exists  a maximal  $B_G(\bL, \la) $-Brauer pair  $ (D, e)$ such that  
 \begin{itemize}
  \item $(Z(\bL)_\ell, e(\la)  ) \unlhd (D, e)$;
  \item $C_D(Z(\bL)_\ell)  \leq  Z(L)_\ell $; and
  \item $D/Z(L)_\ell $ is isomorphic to a subgroup of $W_G(\bL, \la)$.
 \end{itemize}
\end{enumerate} 
\end{thm}

Theorem~\ref{thm:blocks}   provides  a complete solution to  the  block distribution  problem   for unipotent characters. In  other words, it completes    Part  (I)  of the programme  outlined in Section~\ref{subsec:luboro} for $s=1$.  In fact,    Cabanes and Enguehard also give  a solution to  Part  (II). We describe this briefly.  For simplicity, we assume that    $Z(\bG)  $ is connected.  Let  $  r $ be an $\ell $-element   of  $G^*$. The  assumption   that  $\ell \geq 7 $   implies    that    the  centraliser  of $r$  in $\bG^*$ is a Levi subgroup  of $\bG^*$, necessarily $F$-stable.  Duality    between $\bG$ and $\bG^*$  yields a corresponding  $F$-stable Levi subgroup   $\mathbf {C} (r)  \leq  \bG$  and  a  linear character $\hat r$  of $C(r)$.   By Lusztig's parametrisation of characters,   the elements of $\cE(\bG, r)$,  are  precisely  the characters    
$$\epsilon R_{\mathbf{C}(r)}^{\bG}  (\hat r \otimes  \eta), \qquad  \eta   \in \cE (C(r), 1),$$   for some  $\epsilon \in \{ \pm 1 \}$.
Cabanes and Enguehard   show that  to   each    unipotent $d$-cuspidal pair  $ (\bL', \la')$  of $C(r)$   is associated  a   unipotent $d$-cuspidal pair $(\bL,  \la)$ of $G$  such that 
$$[\bL, \bL] = [\bL', \bL'] \qquad  \text{and} \qquad \text{Res}^{\bL}_{[\bL, \bL]^F} \la= \text{Res}^{\bL'}_{[\bL', \bL']^F} \la. $$ Then  $ \epsilon R_{\mathbf{C}(r)}^{\bG}  (\hat r \otimes  \eta) $ belongs  to the block $B_{(\bL, \la)} $ if and only if $\eta  $ lies in the   $d$-Chandra series  of $C(r)$  above   a   $d$-cuspidal pair $(\bL',\la')$   associated to  $(\bL, \la)$.

The condition  $\ell \geq 7 $  in the  above  theorem  can be replaced  by the  weaker condition:  $\ell $ is odd,  good for $\bG$  and  $\ell \geq 5 $    if  $\bG $  has  a  simple component of type $D_4 $ which contributes the  triality  group   $\,^3D_4 (q) $  to $\bG^F$.
In  \cite{E00}  Enguehard   treated     the  unipotent  $\ell$-blocks   for  the  remaining primes.  One obtains   a slightly weaker analogue  of  Theorem~\ref{thm:blocks}.  The main difference is that    the  assignment in  part (b) of the theorem, while still onto,  is no longer one-to-one. In order to  obtain a bijection   one replaces  the set of unipotent $d$-cuspidal pairs    with  a slightly smaller set, namely   the set  of  unipotent $d$-cuspidal pairs  with central  $\ell$-defect.  Enguehard  also  does Part (II)  of the problem  but  only in the case that  the center of $\bG$  is connected.    For    disconnected center groups   the problem is  still open.

\subsection{General Blocks.} 
There has also been a lot of work done to generalise the results of the
previous section to non-unipotent blocks. There are two inter-connected
approaches to this generalisation: (i)  develop  a   non-unipotent $d$-Harish Chandra  theory  (ii) use  Jordan decomposition    to  carry over  unipotent $d$-Harish Chandra theory  to the non-unipotent case.

In  \cite{CE99},    using  a  hybrid  of the  two approaches,   Cabanes and Enguehard  proved an analogue  of  Parts (a)   and (b)  of Theorem~\ref{thm:blocks} for  non-unipotent blocks  as well as a  weak analogue of (c) provided  $\ell $ is odd,  good for $\bG$  and  $\ell \geq 5 $    if  $\bG $  has  a  simple component of type $D_4 $.  They    also  describe defect  group structure.  In  \cite{E08}, Enguehard  showed that    provided  that  $\ell \geq 7$  and the center of $\bG$  is connected, then  block  distribution  of irreducible characters is highly  compatible with   Jordan decomposition.  In particular,  for any semisimple $\ell'$-element  $s$ of  $G^*$,  there is a bijection   $  B \mapsto B_s $    between the blocks of  $\bG$  in $\cE_{\ell} (\bG, s)  $     and the unipotent  blocks of $C_{\bG^*}(s)$   such that  there is a height preserving  bijection between the irreducible characters  of $B$ and those of $B_s$   and   such that  $B$ and $B_s$ have isomorphic
   defect groups (in fact  $B$ and $B_s$  have isomorphic  Brauer categories). In the same paper, Enguehard  also  describes the blocks for classical groups   when $\ell=2 $.  
The paper  \cite{KM13} described  the  block distribution  of $\cE_{\ell} (G, s)$      for $\bG$ simple of exceptional type  and $\ell $  a bad prime.  Combining  all of the  previous   results   with the theorem of Bonnaf\'e and Rouquier,    a uniform  parametrisation  of       blocks  for  all $G$ such that $\bG$ is simple    was given in  \cite{KM15}.

\section{On the other conjectures; open problems}
To end this survey, let us briefly comment on the status of the inductive
conditions for the other local-global conjectures beyond the McKay
conjecture and on some related open problems.

The work on block parametrisation described in these sections has been enough
to verify Brauer's height zero conjecture for quasi-simple groups \cite{KM13}, \cite{KM15}.
For the remaining conjectures, we will need to do much more. For the immediate
future, the two main open problems which need to be resolved for blocks of
finite quasi-simple groups of Lie type are:

\begin{prob}
 Complete the description of non $\ell'$-characters in $\ell$-blocks of
 finite groups of Lie type for small (bad) primes $\ell$.
\end{prob}

\begin{prob}
 Describe the relationship between the $\ell$-block distribution of $\cP(G)$
 and Theorem~\ref{thm:blocks}.
\end{prob}

For the Alperin--McKay Conjecture~\ref{conj:AM} we still do not have
control over the global nor over the local situation in general.
The hope is that we can prove a reduction of the necessary conditions to
so-called quasi-isolated blocks, in the spirit of the Bonnaf\'e--Rouquier
Theorem~\ref{thm:BR}: While this result gives some kind of reduction for the
global situation, we are still missing an analogous local result.

For the Alperin weight conjecture, the following cases have been dealt with so
far, see \cite{Ma14} and \cite{Sch15}:

\begin{thm}[Malle (2014), Schulte (2016)]
 The groups $\fA_n$, $\tw2B_2(q^2)$, $^2G_2(q^2)$, $\tw2F_4(q^2)$,
 $\tw3D_4(q)$ and $G_2(q)$ are AWC good for all primes $\ell$.
\end{thm}

The proof requires the determination of all weights of all radical subgroups;
for exceptional groups of larger rank that seems quite challenging at the
moment. For classical types, it might be possible to use results of
An \cite{An94}.

Here we hope for
\begin{enumerate}
 \item a generic description of weights in terms of $d$-tori and their
  normalisers
 \item a Bonnaf\'e--Rouquier type reduction to a few special situations.
\end{enumerate}

Another ingredient might be the following:

\begin{prob}   \label{prob:decmat}
 Show that the $\ell$-modular decomposition matrices of blocks of
 quasi-simple groups of Lie type are unitriangular.
\end{prob}

This statement might follow, at least in good characteristic, from properties of
generalised Gelfand--Graev characters. If this were true, one could make use
of the result of Koshitani--Sp\"ath \cite{KS15a} mentioned before. The
unitriangularity will be with respect to a suitable subset of $\Irr(B)$:
A linearly independent subset $X\subseteq\Irr(B)$ is called a \emph{basic set
for $B$} if every Brauer character $\vhi\in\IBr(B)$ is an integral linear
combination of the elements of $X$. So in particular $|X|=|\IBr(B)|$.
The following is folklore:

\begin{conj}   \label{conj:basic set}
 Any $\ell$-block of a quasi-simple group of Lie type has a ``natural''
 basic set.
\end{conj}

Geck and Hiss \cite{GH91} exhibited such a basic set when $\ell$ is good for
the underlying algebraic group $\bG$ and does not divide the order of $Z(G)$:
in this case $\cE(G,s)$ is a basic set for the union of blocks $\cE_\ell(G,s)$.
This is yet another situation in which the theories of Brauer and of Lusztig
fit together perfectly. It is known that this statement can no longer hold
when either $\ell$ divides $|Z(G)|$, or when $\ell$ is bad for $\bG$.
In some cases, replacements have been found, but this is still
open in general.

\vskip 1pc
\noindent
{\bf Acknowledgement:} We thank Britta Sp\"ath for her pertinent comments on a
previous version.


\end{document}